\documentclass[11pt]{amsart}
\usepackage{mathrsfs}
\usepackage{threeparttable}
\usepackage{amsfonts}
\usepackage{amsmath}
\usepackage{amssymb}
\usepackage{enumerate}
\usepackage{latexsym}
\usepackage{color}
\usepackage{enumitem}
\usepackage{enumerate}
\usepackage{makecell}
\usepackage{graphicx} 
\usepackage{hyperref}
\usepackage[nameinlink,noabbrev]{cleveref}  
\usepackage{aliascnt}                       
\usepackage{listings}

\def\a{\alpha}    \def\b{\beta}             
       \def\O{{\rm O}}   \def\Ga{\Gamma}    
\def\Syl{\hbox{\rm Syl}}                
\def\qed{\hfill $\Box$}                    
\def\pf{\noindent{\it Proof.}}

\newcommand\ZZ{\mathrm C}        \newcommand\D{\mathrm{D}}        
     
\newcommand\A{\mathrm{A}}        \newcommand\Sy{\mathrm{S}}       
\newcommand\Aut{\mathrm{Aut}}        
\newcommand\Cay{\mathrm{Cay}}         
\newcommand\K{\mathsf{K}}        \newcommand\soc{\mathrm{soc}}    \newcommand\Core{\mathrm{Core}}

\newcommand\N{\mathbf{N}}                 
    
\newcommand\Om{\mathrm{\Omega}}

    \newcommand\SL{\mathrm{SL}}           \newcommand\AGL{\mathrm{AGL}} 
     
\newcommand\PSL{\mathrm{PSL}}  
                
\newcommand\PSU{\mathrm{PSU}} 
               
\newcommand\Sp{\mathrm{Sp}}    \newcommand\PSp{\mathrm{PSp}}         
            
\newcommand\GammaL{\mathrm{\Gamma L}}           
\newcommand\AGammaL{\mathrm{A\Gamma L}}         
         
\newcommand\POmega{\mathrm{P\Omega}}            
    \newcommand\McL{\mathrm{McL}}   
         
\newcommand\E{\mathrm{E}}      \newcommand\F{\mathrm{F}}       \newcommand\G{\mathrm{G}}
       
\newcommand\M{\mathrm{M}}             \newcommand\Co{\mathrm{Co}}
\newcommand\HS{\mathrm{HS}}

\newtheorem{theorem}{Theorem}[section]
\newtheorem{hypothesis}[theorem]{Hypothesis}

\newtheorem{lemma}[theorem]{Lemma}
\newtheorem{definition}{Definition}[section]

\theoremstyle{definition}

\makeatletter
\def\subsection{\@startsection{subsection}{2}{\z@}%
   {1.5ex\@plus1ex \@minus.2ex}%
   {1.5ex \@plus.2ex}%
   {\normalfont\bfseries}}
\makeatother

\begin{document}
\title[Symmetric Cayley graphs]{Prime-valent Symmetric Cayley Graphs of Characteristically Simple Groups}
\thanks{Corresponding author: Hao Yu}
\thanks{The project was partially supported by NSF of Guangxi (2025GXNSFAA069013) and the NNSF of China (12571362, 12601653)}
\author[F. Deng, J.J. Li \and H. Yu]{%
Feng Deng, Jing Jian Li and Hao Yu}
\address{}
\address{School of Mathematics \& Guangxi Base, Tianyuan Mathematical Center in Southwest China \& Guangxi Center for Mathematical Research \& Center for Applied Mathematics of Guangxi (GXU), 
  Guangxi University, Nanning, Guangxi 530004, P. R. China. }
\email{dengfeng@st.gxu.edu.cn (F. Deng); lijjhx@gxu.edu.cn (J.J. Li); haoyu@gxu.edu.cn (H. Yu).}

\begin{abstract}
Let $\Ga$ be a connected prime-valent $X$-arc-transitive Cayley graph of a finite characteristically simple group $G\cong T^k$, where $k\geqslant2$. We obtain a precise structural characterization of such graphs and their arc-transitive automorphism groups. In the cubic case, every connected symmetric Cayley graph of $T^k$, where $T$ is a finite nonabelian simple group, is normal.

\vskip 5pt


\end{abstract}
\maketitle 

\parskip 5pt

\section{Introduction} \label{sec1}
Throughout this paper, all graphs are finite, simple, undirected and connected. We begin with some background and notation used throughout the paper.

Let $\Ga$ be a graph. Denote by $V(\Ga)$, $E(\Ga)$ and $\Aut(\Ga)$ the vertex set, edge set and full automorphism group of $\Ga$, respectively. For $\a\in V(\Ga)$, let $X_{\a} $ be the stabilizer of $\a$ in $X$ and let $\Ga(\a)$ be the neighborhood of $\a$ in $\Ga$, i.e., the set of all points adjacent to $\a$ in $\Ga$. 
For a positive integer $s$, an \textit{$s$-arc} in $\Ga$ is a sequence of vertices $(\a_0,\a_1,\ldots,\a_s)$ such that $\{\a_{i}, \a_{i+1}\}\in E(\Ga)$ for $0\leqslant i\leqslant s-1$ and $\a_{j-1} \neq \a_{j+1}$ for $1\leqslant j\leqslant s-1$. The graph $\Ga$ is said to be \textit{$(X,s)$-arc-transitive} if $X$ acts transitively on vertices and $s$-arcs of $\Ga$ for $X\leqslant \Aut(\Ga)$, and \textit{$(X,s)$-transitive} if it is $(X,s)$-arc-transitive but not $(X,s+1)$-arc-transitive. In particular, $\Ga$ is simply called \textit{$X$-arc-transitive} or \textit{$X$-symmetric} for $s=1$, and if $X=\Aut(\Ga)$, then $X$ is often omitted.

Let $\Ga$ be an $X$-symmetric graph with $N\unlhd X\leqslant\Aut(\Ga)$. The \textit{normal quotient graph} $\Ga_N$ of $\Ga$ relative to $N$ is defined as the graph with vertices the $N$-orbits in $V(\Ga)$, such that two $N$-orbits $X, Y$ are adjacent (joined by an edge) in $\Ga_N$ if and only if at least one $x\in X$ and at least one $y\in Y$ are adjacent in $\Ga$.
By the above definition, the valency of $\Ga_N$ necessarily divides the valency of the graph $\Ga$, and if the two are equal, then the graph $\Ga$ is called a \textit{normal cover of} $\Ga_N$.

For a finite group $G$ and an inverse-closed subset $S$ of $G$ not containing the identity, i.e., $1\not\in S\subset G$ and $S=S^{-1}:=\{s^{-1}|s\in S\}$, the \textit{Cayley graph} $\Cay(G, S)$ of $G$ is the graph with vertex set $G$ such that two vertices $x$ and $y$ are adjacent if and only if $yx^{-1}\in S$. It is easy to see that $\Cay(G, S)$ is connected if and only if $S$ generates the group $G$. The graph $\Ga=\Cay(G, S)$ is said to be \textit{X-normal} if $G\unlhd X\leqslant\Aut(\Ga)$, in particular, if $X=\Aut(\Ga)$, then $\Ga$ is simply said to be \emph{normal}. Let $\Aut(G, S)=\{x\in\Aut(G)|S^x=S\}$. For $1\in G$, It follows directly from the definition that $\Aut(G, S)\leqslant (\Aut(\Ga))_1$. It is well known that  $\N_{\Aut(\Ga)}(G)=G\rtimes\Aut(G, S)$.

The study of symmetric graphs, which aims to characterize and construct such graphs, has attracted considerable attention in the literature. It traces back to Tutte's work on cubic symmetric graphs in \cite{T1947}, which established that finite 6-arc-transitive graphs do not exist. Further, Weiss extended this result by proving that the only possible values for $s$-transitivity are $s\leqslant 5$ and $s=7$ (\cite{W1981}).
With certain limiting conditions, the classification of symmetric graphs is also well-studied, such as the classification of symmetric graphs of some specific orders: $kp$ (\cite{C1971,C1987,PWX1993,PX1993,WX1993}), where $p\geqslant 3$ is prime and $k$ equals 1 or is a prime distinct from $p$.

A substantial amount of work has been devoted to symmetric Cayley graphs of finite nonabelian simple groups. For small prime valencies, cubic symmetric Cayley graphs were studied in \cite{XFWX2005}, pentavalent symmetric Cayley graphs in \cite{ZF2010}, and heptavalent symmetric Cayley graphs in \cite{LM2021,LMZ2022,ZFYW2025}. The prime-valent problem was considered more generally in \cite{YFZC2021} for solvable vertex stabilizers and in \cite{LLarXiv2021} for nonsolvable vertex stabilizers. Together, these works give rather precise structural characterizations of prime-valent symmetric Cayley graphs of finite nonabelian simple groups.
A natural question arises: What are the properties of symmetric Cayley graphs of characteristically simple groups with prime valency? This paper is devoted to a series of studies on this question. A connected 2-valent graph is a cycle. Thus we assume that $p$ is an odd prime.

For the sake of convenience, we adopt the following conventions in this paper.

\begin{hypothesis}\label{hypothesis}
Let $\Ga=\Cay(G, S)$ be a connected $p$-valent $X$-arc-transitive Cayley graph of finite group $G=T_1\times\cdots\times T_k\cong T^k$, where each $T_i\cong T$ is simple, integer $k\geqslant2$, $p$ is an odd prime and $G<X\leqslant\Aut(\Ga)$. When $X$ is not quasiprimitive on $V(\Ga)$, let $N$ be a maximal intransitive normal subgroup of $X$, $\overline{\a}\in V(\Ga_N)$, $\overline{X}=X/N$, $\overline{G}=GN/N\cong T^\ell$ for some $1\leqslant\ell\leqslant k$ and $\soc(\overline{X})\cong Q^r$, where $Q$ is finite simple group.
\end{hypothesis}

We give a classification of connected prime-valent symmetric Cayley graphs of finite characteristically simple groups, which will be presented in the following theorem.

\begin{theorem} \label{thm:main}
Let $\Ga$ be a graph described in Hypothesis \ref{hypothesis}. If $X$ is quasiprimitive, then either $\Ga\cong\K_{2^k}$, or $G=\soc(X)$ whenever $X$ is type of twisted wreath, and $\Ga$ is isomorphic to a Cayley graph of $\soc(X)\cong\ZZ_2^k$ whenever $X$ is type of affine. Otherwise, $\overline{X}$ is quasiprimitive on $V(\Ga_N)$ and  one of the following holds. 
\begin{enumerate}[leftmargin=1.2em] 
   \item $\Ga_N$ is a Cayley graph of $\overline{G}\cong\ZZ_2^\ell$. 
   \item $\Ga_N$ is a Cayley graph of $\soc(\overline{X})$ with $\overline{X}$ being type of twisted wreath or affine, and for the latter case, $\Ga_N\cong\K_8$. 
   \item $\overline{X}$ is of almost simple type and either $\Ga_N$ is a complete graph, or $\overline{G}=T$ such that either $\overline{G}\unlhd \overline{X}$ or the pair $(\soc(\overline{X}), T)$ is listed in Tables \ref{tab:simple-nonclassical} and \ref{tab:simple-classical-nonsolv}.
   \item $\overline{X}$ is of product action type and either $\overline{G}\unlhd \overline{X}$ or $(Q,\,T,\,r)$ is listed in Table \ref{tab:PA-final}.
\end{enumerate}
\end{theorem}

The theorem above provides a structural characterization of prime-valent symmetric Cayley graphs of finite characteristically simple groups. We next consider the cubic case and obtain the following stronger conclusion.

\begin{theorem}\label{cubic}
If $\Ga$ is a connected cubic symmetric Cayley graph of $G\cong T^k$ for $T$ a finite nonabelian simple group and $k\geqslant2$ integer, then $\Ga$ is normal.
\end{theorem}

After this introductory section, some notation, basic definitions and useful facts will be given in Section \ref{sec2}. The quasiprimitive case will be treated in Section \ref{sec3}, while the non-quasiprimitive case will be considered in Section \ref{sec4}. Theorems \ref{thm:main} and \ref{cubic} will be proved in Section \ref{sec5}.

\section{Preliminaries}\label{sec2}

The notation and terminology used in the paper are standard and can be found in \cite{PS2018}. For example, we use $T^k$ to denote the direct product of $k$ copies of group $T$. For a group $X$ and a prime $r$, denote the largest normal $r$-subgroup of $X$ by $\O_r(X)$. For a positive integer $n$ and a prime $p$, let $v_p(n)$ denote the largest nonnegative integer $a$ such that $p^a\mid n$, and let $\Pi(n)$ denote the set of prime divisors of $n$.
\subsection{groups}
We recall the following from \cite[Definition 2.2]{DS2020} regarding subgroups of direct products:
\begin{definition} \label{defi}
Let $M_1,\,\ldots,\,M_r$ be groups and $N=M_1\times\cdots\times M_r$ and $H\leqslant N$. Consider, for $i\in \{1,\ldots,r\}$, the projections $\pi_i: N\to M_i$ such that $(x_1,\ldots,x_r)\mapsto x_i$ for $x_i\in M_i$ and $1\leqslant i\leqslant r$. Then $H$ is called a {\it strip} of $N$ if $H\neq1$, for each $i$, either the restriction of $\pi_i$ to $H$ is injective or $\pi_i(H)=1$. Set $\mathcal{P}(H)=\{M_i \mid \pi_i(H)\neq1\}$. Two strips $H$ and $K$ are said to be {\it disjoint} if $\mathcal{P}(H)\cap\mathcal{P}(K)=\emptyset$. $H$ is a {\it subdirect subgroup} of $N$ if $\pi_i(H)=M_i$ for each $i$. $H$ is a {\it diagonal subgroup} of $N$ if the restriction of $\pi_i$ to $H$ is injective for each $i$. Further, $H$ is a {\it full diagonal subgroup} of $N$ if $H$ is both a subdirect and a diagonal subgroup of $N$. 
\end{definition}

For a permutation group $X$ acting on $\Om$, it is \textit{quasiprimitive} if its every nontrivial normal subgroup is transitive on $\Om$. The following is a property of quasiprimitive groups.

\begin{lemma}{\cite[Theorem 1.2]{DS2020}}\label{subgroup}
Let $X$ be a finite quasiprimitive permutation group acting on $\Om$ with nonabelian socle and let $G$ be a transitive nonabelian characteristically simple subgroup of $X$. Then $G\leqslant \soc(X)$.
\end{lemma}

For normal subgroup of a direct product of nonabelian simple groups, we obtain the following simple result. 
\begin{lemma}\label{obs} \cite[Theorem 4.16]{PS2018}
Let $G$ be a characteristically simple group say $G=T^k$ for some nonabelian simple group $T$ and $k\geqslant2$ integer, and $K\unlhd G$. Then $K=T^l$ for $l\leqslant k$.
\end{lemma}

Let $q>1$ and $n>1$ be integers. A prime $z$ is called a \emph{primitive prime divisor} of $q^n-1$ if $z\mid q^n-1$ but $z\nmid q^i-1$ for every $1\leqslant i<n$. 
\begin{lemma}\label{lem:ppd} \cite[Theorem IX.8.3(Zsigmondy's theorem)]{HB1982}
$q^n-1$ has a primitive prime divisor except $(q,n)=(2,6)$, or $n=2$ and $q+1$ is a power of $2$. In particular, if $z$ is a primitive prime divisor of $q^n-1$, then $n\mid z-1$. 
\end{lemma}

Using the classification of nonabelian simple groups with the same set of prime divisors, we obtain the following result.

\begin{lemma}\label{lem:equal-prime-divisors}
Let $Q$ and $T$ be finite nonabelian simple groups with $T<Q$ and $\Pi(|T|)=\Pi(|Q|)$. Suppose that there exist an integer $r\geqslant2$ and a nontrivial subgroup $R\leqslant Q^r$ with $R\lesssim Q$ such that $|Q|^r=|T|^r|R|/|T^r\cap R|$ and $|R|<|Q|$. Then $(Q,T,r)$ is listed in Tables \ref{tab:2.1} and \ref{tab:2.2}.
\end{lemma}
\begin{table}[htbp]
\centering
\renewcommand{\arraystretch}{1.3}
\setlength{\tabcolsep}{10pt}
\caption{Infinite cases for triples $(Q,T,r)$ such that $\Pi(|T|)=\Pi(|Q|)$}
\label{tab:2.1}
{\scriptsize
\begin{tabular}{|c|c|c|c|c|}
\hline
Row & $Q$ & $T$ & $r$ & Remark\\
\hline
1 & $\A_c$ & $\A_k$ &   $\leqslant\min_{p\mid n}\lfloor\frac{v_p(|Q|)}{v_p(n)}\rfloor$ & $5\leqslant k<c$, $(k,c]$ 
no prime, $n=\frac{c!}{k!}$\\
\hline
2 & $\PSp_{2m}(q)$ & $\POmega^-_{2m}(q)$ &  2  & $m, q$ are even, $(m,q)\neq(2,2)$\\
\hline
3 & $\POmega_{2m+1}(q)$ & $\POmega^-_{2m}(q)$ &  2 & $m$ is even, $q$ is odd\\
\hline
4 & $\POmega^+_{2m}(q)$ & $\POmega_{2m-1}(q)$&  2  & $m\geqslant4$ is even\\
\hline
5 & $\PSp_4(q)$ & $\PSL_2(q^2)$ &  2 & $q>2$\\
\hline
\end{tabular}}
\end{table}

\begin{table}[htbp]
\centering
\renewcommand{\arraystretch}{1.3}
\setlength{\tabcolsep}{10pt}
\caption{Specific cases for triples $(Q,T,r)$ such that $\Pi(|T|)=\Pi(|Q|)$}
\label{tab:2.2}
{\scriptsize
\begin{tabular}{|c|c|c|c|c|c|}
\hline
$Q$ & $T$ & $r$ & $Q$ & $T$ & $r$ \\
\hline
$\PSL_6(2)$ & $\PSL_5(2)$ & 2 &   $\M_{11}$ & $\PSL_2(11)$ & 2   \\
\hline
$\PSU_4(2)$ & $\A_6$      & 2 &  $\M_{12}$ & $\M_{11}$ & 2,\,3    \\
\hline
$\Sp_6(2)$  & $\A_8$      & 2 &  $\M_{24}$ & $\M_{23}$ & 2,\,3   \\
\hline
$\POmega^+_8(2)$ & $\A_9$ & 2 &  & &\\
\hline
\end{tabular}}
\end{table}

\pf\,\, Assume that $Q$ and $T$ are nonabelian simple groups such that $T<Q$ and $\Pi(|T|)=\Pi(|Q|)$. The possible pairs $(Q,T)$ are given by \cite[Table 10.7]{LPS2000} via \cite[Corollary 5]{LPS2000}.  We now turn to determining these explicitly.

Step I. Restricting the possible pairs $(Q, T)$ using the condition where $\Pi(|T|)=\Pi(|Q|)$ for nonabelian simple groups $Q$ and $T$ such that $T<Q$.

We now determine the pairs $(Q,T)$ satisfying $\Pi(|Q|)=\Pi(|T|)$. By \cite[Table 10.7]{LPS2000}, $T$ is contained in one of the subgroups listed there. For each row, we compare the prime divisors of the orders of $Q$ and the possible nonabelian simple subgroups $T$. If the subgroup containing $T$ is solvable, that row is discarded immediately. Otherwise, the order formulae for finite simple groups, together with the parameter restrictions in \cite[Table 10.7]{LPS2000}, determine when $\Pi(|Q|)=\Pi(|T|)$. This yields precisely the pairs listed in Tables \ref{tab:equal-primes1} and \ref{tab:equal-primes2}. We illustrate the argument with two cases.

Suppose that $Q=\A_c$ and $\A_k\unlhd M\leqslant\Sy_k\times\Sy_{c-k}$ such that $T\leqslant M$ is nonabelian simple, as in Row 1 of \cite[Table 10.7]{LPS2000}. Since the prime divisors of $|\A_n|$ are exactly the primes not exceeding $n$, the equality $\Pi(|\A_k|)=\Pi(|\A_c|)$ holds if and only if there is no prime $r$ with $k<r\leqslant c$. Thus $(Q,T)=(\A_c,\A_k)$, where $5\leqslant k<c$ and $(k,c]$ contains no prime. 

Assume that $Q=\PSL_2(8)$, as in Row 8 of \cite[Table 10.7]{LPS2000}. Then $T$ is contained in either $\ZZ_7.\ZZ_2$ or $P_1$. By \cite[Proposition 4.1.17]{KL1990}, $P_1\cong\ZZ_2^3\rtimes\ZZ_7$. Both groups are solvable, and so neither contains a nonabelian simple subgroup. Thus Row 8 gives no pair $(Q,T)$.

Applying this criterion to the remaining rows of \cite[Table 10.7]{LPS2000} gives Tables \ref{tab:equal-primes1} and \ref{tab:equal-primes2}
\begin{table}[htbp]
\centering
\renewcommand{\arraystretch}{1.3} 
\setlength{\tabcolsep}{10pt} 
\caption{Infinite cases for pairs $(Q,T)$ such $\Pi(|T|)=\Pi(|Q|)$}
\label{tab:equal-primes1}
{\scriptsize\begin{tabular}{|c|c|c|c|}
\hline
$Q$ & $T$ & Remark\\
\hline
$\A_c$ & $\A_k$ & $5\leqslant k<c$, and $(k,c]$ contains no prime\\
\hline
$\PSp_{2m}(q)$ & $\POmega^-_{2m}(q)$ & $m$ and $q$ are even, and $(m,q)\neq(2,2)$\\
\hline
$\POmega_{2m+1}(q)$ & $\POmega^-_{2m}(q)$ & $m$ is even and $q$ is odd\\
\hline
$\POmega^+_{2m}(q)$ & $\POmega_{2m-1}(q)$ & $m\geqslant4$ is even\\
\hline
$\PSp_4(q)$ & $\PSL_2(q^2)$ & $q>2$\\
\hline
\end{tabular}}
\end{table}

\begin{table}[htbp]
\centering
\renewcommand{\arraystretch}{1.3} 
\setlength{\tabcolsep}{10pt} 
\caption{Specific cases for pairs $(Q,T)$ such that $\Pi(|T|)=\Pi(|Q|)$}
\label{tab:equal-primes2}
{\scriptsize\begin{tabular}{|c|c|c|c|}
\hline
 $Q$ & $T$ &  $Q$ & $T$    \\
\hline
$\PSL_6(2)$ & $\PSL_5(2)$ &  $\PSU_5(2)$, $\M_{11}$ & $\PSL_2(11)$ \\
\hline
$\PSU_4(2)$ & $\A_5$, $\A_6$ & $\PSU_6(2)$, $\HS$, $\McL$ & $\M_{22}$  \\
\hline
$\Sp_6(2)$ & $\A_7$, $\A_8$ &  $\G_2(3)$ & $\PSL_2(13)$  \\
\hline
$\POmega^+_8(2)$ & $\A_7$, $\A_8$, $\A_9$ &   ${}^2\F_4(2)'$ & $\PSL_2(25)$  \\
\hline
$\PSU_3(3)$ & $\PSL_2(7)$ &   $\M_{12}$ & $\M_{11}$, $\PSL_2(11)$  \\
\hline
$\PSU_3(5)$, $\PSp_4(7)$ & $\A_7$ &   $\M_{24}$, $\Co_2$, $\Co_3$ & $\M_{23}$  \\
\hline
$\PSU_4(3)$ & $\PSL_3(4)$, $\A_7$ &   &\\
\hline
\end{tabular}}
\end{table}

Step II. Derivation of the candidate pairs $(Q, T)$ based on the condition where $|Q|^r=|T|^r|R|/|T^r\cap R|$.

Suppose that there exists a nontrivial subgroup $R\leqslant Q^r$ with $R\lesssim Q$, $|R|<|Q|$, and $|Q|^r=|T|^r|R|/|T^r\cap R|$ for some integer $r\geqslant2$. Let $n=|Q:T|$. Then $n^r=|R:T^r\cap R|$, and so $n^r\mid|R|$. Since $R\lesssim Q$, we have $v_p(|R|)\leqslant v_p(|Q|)$ for every prime $p$. Thus $r\,v_p(n)\leqslant v_p(|Q|)$ for every $p\mid n$, and so $r\leqslant\min_{p\mid n}\left\lfloor\frac{v_p(|Q|)}{v_p(n)}\right\rfloor$. We apply this criterion to all pairs $(Q,T)$ listed in Tables \ref{tab:equal-primes1} and \ref{tab:equal-primes2}. The resulting cases are recorded in Tables \ref{tab:2.1} and \ref{tab:2.2}. We give one case to illustrate the calculation.

Let $Q=\PSp_{2m}(q)$ and $T=\POmega^-_{2m}(q)$ for $m$ and $q$ even. It follows from $|Q|=q^{m^2}\prod_{i=1}^{m}(q^{2i}-1)$ and $|T|=q^{m(m-1)}(q^m+1)\prod_{i=1}^{m-1}(q^{2i}-1)$ that $n=|Q:T|=q^m(q^m-1)$. Applying the Lemma \ref{lem:ppd}, one yields that there exists a primitive prime divisor $z$ of $q^m-1$, except when $(q,m)=(2,6)$, or when $m=2$ and $q+1$ is a power of $2$. Since $q$ is even, the latter case does not occur. Suppose first that $(q,m)\neq(2,6)$. 
If $z\mid q^{2i}-1$ for $1\leqslant i\leqslant m$, then $m\mid2i$, and so $i=m/2$ or $m$. Derived from $(q^m-1,q^m+1)=1$, we conclude that $z\nmid q^m+1$ and $v_z(|Q|)=2v_z(q^m-1)$, and further, $v_z(|Q|)=2v_z(n)$ as $v_z(n)=v_z(q^m-1)$. It deduced from $n^r\mid|Q|$ for $r\geqslant2$ integer that $r=2$. Now suppose that $(q,m)=(2,6)$. Direct calculation yields that $v_7(|Q|)=2$ and $v_7(n)=1$. Still since $n^r\mid|Q|$, one yields that $r=2$.

For the remaining pairs in Tables \ref{tab:equal-primes1} and \ref{tab:equal-primes2}, the same argument is applied by choosing a suitable prime divisor of $n=|Q:T|$ and comparing its valuations in $n$ and $|Q|$. The resulting cases are listed in Tables \ref{tab:2.1} and \ref{tab:2.2}, proving the lemma.  \qed

\subsection{graphs}

Let $X\leqslant\Aut(\Ga)$, where $\Ga$ is a graph, such that $X$ is transitive on $V(\Ga)$. Then $\Ga$ is said to be \textit{$X$-locally primitive} if $X_\a$ is primitive on $\Ga(\a)$, where $\a\in V(\Ga)$. In particular, $\Ga$ is a \textit{locally primitive graph} if $X=\Aut(\Ga)$. Since every symmetric graph of prime valency is locally primitive, the following result on locally primitive graphs is useful.

\begin{lemma}\cite[Theorem 5]{HP2018}\label{NQ}
Let $\Ga$ be a connected $X$-vertex-transitive and locally-primitive graph, and let $N\unlhd X$ have at least three orbits on $V(\Ga)$. Then the following hold. 
\begin{enumerate}[leftmargin=1.2em]  
\item [(1)] $N$ is semiregular on $V(\Ga)$, $X/N\leqslant\Aut(\Ga_N)$, $\Ga_N$ is $X/N$-locally-primitive, and     $\Ga$ is an $N$-cover of $\Ga_N$; 
\item [(2)] $\Ga$ is $(X,s)$-arc-transitive if and only if $\Ga_N$ is $(X/N,s)$-arc-transitive for $1\leqslant     s\leqslant 5$ or $s=7$;
\item [(3)] $X_\a\cong (X/N)_v$, where $\a\in V(\Ga)$ and $v\in V(\Ga_N)$;
\item [(4)] If $X$ has a normal subgroup $M$ which is contained in $N$, then $\Ga_M$ is an $X/M$-locally-primitive $N/M$-cover of $\Ga_N$.
\end{enumerate}
\end{lemma}

It is known that a graph $\Ga$ is $X$-symmetric if and only if $X$ is vertex-transitive and $X_\a$ is transitive on $\Ga(\a)$ for some $\a\in V(\Ga)$. Denote by $X_\a^{\Ga(\a)}$ the permutation group induced by the action of $X_\a$ on $\Ga(\a)$, and by $X^{[1]}_\a$ the kernel of $X_\a$ acting on $\Ga(\a)$. Then $X^{\Ga(\a)}_\a\cong X_{\a}/X^{[1]}_\a$. At present, the structure of vertex stabilizers of connected symmetric graphs of prime valency has been completely determined (see \cite[Propositions 2--5]{DM1980}, \cite[Theorem 1.1]{GF2012}, \cite[Theorem 1.1]{GLH2016}, \cite[Theorem 2.1]{GHX2015}, \cite[Theorem 1.2]{LLL2018}), which lead to the following conclusion.

\begin{lemma}\label{power}
Let $\Ga$ be a connected $X$-symmetric graph of odd prime valency $p$. Then $v_p(|X_\a|)=1$ for every $\a\in V(\Ga)$. 
\end{lemma}

\section{Vertex quasiprimitive case}\label{sec3}

In this section, we always let $\Ga$ be a connected $p$-valent $(X,s)$-transitive Cayley graph of characteristically simple group $G=T_1\times\cdots\times T_k\cong T^k$, where each $T_i\cong T$ is simple, $s$ is a positive integer, $p$ is an odd prime, $k\geqslant 2$ is an integer and $G\leqslant X\leqslant \Aut(\Ga)$. Let $\a\in V(\Ga)$. Suppose that $X$ is quasiprimitive on $V(\Ga)$. Then $X$ is of one of the following five types: affine, almost simple, simple diagonal, product action or twisted wreath (see \cite[Theorem 1]{P1993} for more details). The quasiprimitive case under Hypothesis \ref{hypothesis} is given by the following theorem.

\begin{theorem}\label{thm:quasiprimitive}
Let $\Ga$ be as in Hypothesis \ref{hypothesis}. If $X$ is quasiprimitive, then one of the following holds.
\begin{enumerate}[leftmargin=1.2em]  
\item [(1)] $X$ is of affine type, and $\Ga$ is a Cayley graph of $G\cong\soc(X)\cong\ZZ_2^k$;
\item [(2)] $X$ is of almost simple type, and $\Ga\cong\K_{2^k}$; 
\item [(3)] $X$ is of twisted wreath type, $\Ga=\Cay(\soc(X), S)$, where $S$ consists of involutions. 
\end{enumerate}
\end{theorem}

Next, we prove the above theorem by establishing some technical lemmas.

\begin{lemma}\label{affine}
Suppose that $X$ is of affine type. Then $T\cong\ZZ_2$, $\Ga$ is isomorphic to a Cayley graph of $\soc(X)\cong\ZZ_2^k$. %Moreover, $s\leqslant2$, and $\Ga\cong\K_{2^k}$, $\Box_p$ or $\Ga({\rm C}_{23})$ for $s=2$.
\end{lemma}
\pf\, Suppose that $X$ is of affine type. Then $\soc(X)\cong\ZZ_\ell^d$ for some prime $\ell$ and positive integer $d$, and $\soc(X)$ is regular on $V(\Ga)$. Since $G=T^k$ is also regular on $V(\Ga)$, we have $|T|^k=|G|=|\soc(X)|=\ell^d$. Hence $T\cong\ZZ_\ell$ and $d=k$, so $G\cong\ZZ_\ell^k$. Since $\Ga$ has odd prime valency $p$, the equality $p|V(\Ga)|=2|E(\Ga)|$ implies $\ell=2$. Thus $G\cong\soc(X)\cong\ZZ_2^k$ and $\Ga$ is isomorphic to a Cayley graph of $\soc(X)\cong\ZZ_2^k$. Hence Lemma \ref{affine} holds.\qed

\begin{lemma}\label{almost simple}
Suppose that $X$ is of almost simple type. Then $G\cong \ZZ_2^k$ and $\Ga\cong\K_{2^k}$.
\end{lemma}
\pf\, Suppose that $G$ is solvable, that is, $G\cong\ZZ_{\ell}^k$ for some prime $\ell$. The equation $p\cdot {\ell}^k=2|E(\Ga)|$ with odd prime $p$ forces $\ell=2$. So $T\cong\ZZ_2$ and $G\cong\ZZ_2^k$, which implies that $X$ contains an abelian regular subgroup $G$. Then it follows from \cite[Corollary 3.4]{L2003} that $X$ is primitive. 
By \cite[Theorem 1.1]{L2003}, $X\cong\A_{2^k}$ or $\Sy_{2^k}$. Since $X$ is of almost simple type, we get that $k\geqslant 3$. Then $X_\a\cong\A_{2^k-1}$ or $\Sy_{2^k-1}$ which is nonsolvable. Combing with \cite[Propositions 2-5]{DM1980}, \cite[Theorem 1.1]{GF2012}, \cite[Theorem 1.1]{GLH2016}, \cite[Theorem 2.1]{GHX2015} and \cite[Theorem 1.2]{LLL2018}, we obtain that $s\geqslant 2$. Thus $|V(\Ga)|=|G|=\ZZ_2^k$ and $X\in\{\A_{2^k}, \Sy_{2^k}\}$ that $X$ is 2-transitive on $V(\Ga)$, and so $\Ga\cong\K_{2^k}$ as $\Ga$ is connected, as desired.

\textbf{Claim:} $G$ is solvable. 

Assume that $G$ is nonsolvable. Then $G$ is a transitive nonabelian characteristically simple subgroup of $X$. Let $L=\soc(X)$. The Lemma \ref{subgroup} shows that $G\leqslant L$. We assert that $G$ is core free in $X$, i.e., $\Core_X(G)=1$. On the contrary, assume that $\Core_X(G)\neq 1$. Since $X$ is quasiprimitive, $G$ is regular on $V(\Ga)$ and $\Core_X(G)$ is a normal subgroup of groups $X$ and $G$, we yield that $\Core_X(G)$ is transitive on $V(\Ga)$ and so $G=\Core_X(G)$ and $G\unlhd X$. We deduce from $X$ is of almost simple type that $L\unlhd G\leqslant X$ and so $L=G$ is a nonabelian simple group, a contradiction arisen as $G=T^k$ for $T$ nonabelian simple and $k\geqslant2$ integer. Hence the assertion holds. 

We asset that $X_\a$ is solvable. On the contrary, we suppose that $X_\a$ is nonsolvable. Applying the \cite[Theorem 1.2]{LWX2023}, one yields that all the candidates of triple $(X, G, X_\a)$ are listed in \cite[Table 1]{LWX2023}. Detail checking all the candidates of $G$, since $G=T^k$ (with $k\geqslant2$) is a nonabelian characteristically simple group, we can exclude the cases listed in lines 2-23 of \cite[Table 1]{LWX2023}. Now that the line 1 holds, that is $(X, G, X_\a)=(\A_m.\mathcal{O}, [m].\mathcal{O}_1, \A_{m-1}.\mathcal{O}_2)$ or $(\A_m.\mathcal{O}, \A_{m-1}.\mathcal{O}_2, [m].\mathcal{O}_1)$, where $[m]$ is a regular permutation group on $m$ points, $|G|=m\geqslant60$, $\mathcal{O}_1, \mathcal{O}_2\leqslant\mathcal{O}\leqslant\ZZ_2$ with $|\mathcal{O}|=|\mathcal{O}_1||\mathcal{O}_2|$. If $G$ is isomorphic to $A_{m-1}.\mathcal{O}_2$ or $[m].2$, then it is not the form of $G$. It follows that $(X, G, X_\a)=(\A_m.\mathcal{O}, [m], \A_{m-1}.\mathcal{O})$ and so $\Ga$ is a complete graph of valency $p=m-1=|G|-1$. Since $|G|=|T|^k$ for $k\geqslant2$ and $T$ nonabelian simple, one yields that $p=m-1=|T|^k-1$, a contradiction arisen as $p$ is prime and so the assertion holds. 

Since $X=G X_\a$ with $G\cap X_\a=1$, $G\leqslant L$, $G=T^k$ for some nonabelian simple group $T$ and integer $k\geqslant 2$, by \cite[Theorem 1.1]{LX2022}, one yields that the candidates for triple $(L, L_\a, G)$ lie in \cite[Table 1.1]{LX2022}. Next, we will eliminate all 9 cases successively.

The first line shows that $L=\PSL_n(q)$, $L_\a=X_\a\cap L\leqslant\frac{q^n-1}{(q-1)(n, q-1)}\rtimes\ZZ_n$ and $\E_q^{n-1}\rtimes\SL_{n-1}(q)\unlhd G$, where $q$ is a prime power and $n$ is a  positive integer. Since $G$ is a characteristic group, applying the Lemma \ref{obs}, we yield that $\E_q^{n-1}\rtimes\SL_{n-1}(q)$ is also a character simple group, a contradiction. Hence the first line of \cite[Table 1.1]{LX2022} is ruled out. 

For line 2 of \cite[Table 1.1]{LX2022}, one yields that $L=\PSL_4(q)$ and $\PSp_4(q)\unlhd G$. If $q=2$, then $G\unrhd\PSp_4(2)\cong\Sy_6$ is not characteristic simple, which contradicts the result of the Lemma \ref{obs}. Hence, $q\geqslant 3$ and $\PSp_4(q)$ is nonabelian simple. Using the Lemma \ref{obs} again, one yields that $T\cong\PSp_4(q)$. Comparing the orders $|L|$ and $|T|^2$, one yields that
\begin{equation*}
\frac{|L|}{|T|^2}=\frac{|\PSL_4(q)|}{|\PSp_4(q)|^2}=\frac{(2, q-1)^2(q^3-1)}{q^2(q^4-1)(q^2-1)(4, q-1)}<1.
\end{equation*}
Hence $|L|<|T|^2\leqslant|T|^k=|G|$ for $k\geqslant2$, a contradiction arisen as $G\leqslant L$.

For each of the remaining rows of \cite[Table 1.1]{LX2022}, a direct comparison of the orders of $G$ and $L$ gives $|G|>|L|$, a contradiction. Thus none of the nine cases can occur. 

This proves the claim and completes the proof of Lemma \ref{almost simple}. \qed

In what follows, we provide a necessary condition for $G$ to be solvable.

\begin{lemma}\label{sovable}
$G$ is solvable if and only if $X$ is of affine type or almost simple type.
\end{lemma}
\pf\, Based on the results of the Theorems \ref{affine} and \ref{almost simple}, we find that the sufficiency of this theorem is evident. Now we will prove the necessity of this theorem. 

Since $G$ is solvable, we get that $G\cong\ZZ_{\ell}^k$ for some prime $\ell$. The equation $p\cdot {\ell}^k=2|E(\Ga)|$ with odd prime $p$ forces $\ell=2$. Thus $G\cong\ZZ_2^k$, which implies that $X$ contains an abelian regular subgroup $G$, it follows from \cite[Corollary 3.4]{L2003} that $X$ is primitive. If $\Ga$ is $(X,2)$-arc-transitive, applying the \cite[Lemma 3.3]{LP2008}, one yields that $X$ is of affine type or almost simple type, as desired. Now assume that $s=1$, that is, $\Ga$ is $(X, 1)$-transitive. 

Now let's assume in reverse that $X$ is neither an affine group nor an almost simple group. Then, by \cite[Theorem 1.1]{L2003}, $X=(K_1\times\cdots\times K_r).O.P$ and $G=G_1\times\cdots\times G_r$, where $r\geqslant2$, $|G_i|=2^a$, and each $K_i$ is isomorphic to $\A_{2^a}$ or $\Sy_{2^a}$. 
Moreover, $G_i<K_i$ for every $i$, and $k=ar$. 
If $a=1$, then $G_i\cong\ZZ_2\cong K_i$, a contradiction. 
Suppose that $a=2$. Then $G_i\cong\ZZ_2^2$ acts regularly on four points. In both $\A_4$ and $\Sy_4$, the unique regular Klein four subgroup is the normal subgroup consisting of the identity and the three double transpositions. Hence $G_i=\O_2(K_i)$ for every $i$. Thus $G=\O_2(K_1\times\cdots\times K_r)$ is characteristic in $K_1\times\cdots\times K_r\unlhd X$, so $G\unlhd X$. Since $G$ is regular, $X$ is of affine type, a contradiction. Thus $a\geqslant3$.

Combining \cite[Propositions 2--5]{DM1980}, \cite[Theorem 1.1]{GF2012}, \cite[Theorem 1.1]{GLH2016}, \cite[Theorem 2.1]{GHX2015} and \cite[Theorem 1.2]{LLL2018}, we have $X_\a\cong(\ZZ_p\rtimes\ZZ_m)\times\ZZ_n$, where $m$ is a proper divisor of $p-1$ and $n\mid m$. Since $G$ is regular on $V(\Ga)$, we have $X=GX_\a$ and $G\cap X_\a=1$. Hence $|X|=|G||X_\a|=2^kpmn$. As $m,n<p$, it follows that $v_p(|X|)=1$. 
If $p\mid|K_i|$ for some $i$, then $p\mid|K_j|$ for every $j$, because $\A_{2^a}$ and $\Sy_{2^a}$ have the same set of prime divisors. Hence $v_p(|X|)\geqslant r\geqslant2$, a contradiction. Thus $p\nmid|K_i|$ for every $i$.

Recall that $K_i\cong\A_{2^a}$ or $\Sy_{2^a}$, where $a\geqslant3$. 
Then $|K_i|/2^a\geqslant(2^a-1)!/2$. Thus $|K_1\times\cdots\times K_r|/|G|\geqslant((2^a-1)!/2)^r$. From $|X|=|K_1\times\cdots\times K_r||O||P|=2^kpmn$ and $p\nmid|K_i|$, it follows that $((2^a-1)!/2)^r\leqslant mn$. Note that $n\mid m$ and $m$ is a proper divisor of $p-1$. Then $mn\leqslant m^2< p^2\leqslant(2^k-1)^2<2^{2k}=2^{2ar}$. On the other hand, for $a\geqslant3$, $(2^a-1)!/2>2^{2a}$, and hence $((2^a-1)!/2)^r>2^{2ar}$. This contradiction proves that $X$ is of affine or almost simple type.\qed

Now we will deal with the remaining quasiprimitive types, for which $G$ is nonsolvable by Lemma \ref{sovable}.

\begin{lemma}\label{simple diagonal}
$X$ is not of simple diagonal action type. 
\end{lemma}
\pf\, Suppose that $X$ is of simple diagonal action type. Let $\a\in V(\Ga)$ and $\soc(X)=L^n$ for some nonabelian simple group $L$ and positive integer $n$. Then $(\soc(X))_\a\cong L$, so $X_\a$ is insolvable. Combining with \cite[Theorem 1.1]{GF2012}, \cite[Theorem 1.1]{GLH2016} and \cite[Theorem 1.2]{LLL2018}, we conclude that the graph $\Ga$ is $(X,2)$-arc-transitive. However, derived from \cite[Theorem 2]{P1993}, we yield that the group $X$ cannot be of simple diagonal action type in this case, a contradiction. \qed

\begin{lemma}\label{product action}
$X$ is not of product action type.
\end{lemma}
\pf\, 
Suppose that $X$ is of product action type. Then, by \cite[Section 2]{P1993}, $X$ is a subgroup of $H\wr\Sy_n$, where $H$ is a quasiprimitive permutation group that is either of almost simple type or simple diagonal action type. Let $N=\soc(X)=L_1\times \ldots\times L_n\cong L^n$, where $n\geqslant2$ is an integer, each $L_i\cong L$ is nonabelian simple. Then $N$ is a normal subgroup of $X$ and $1\neq N_\a=N\cap X_\a\trianglelefteq X_\a$ for any $\a\in V(\Ga)$ (see \cite[Section 2, Remark 2.1]{P1993} for example). 

Since $\Ga$ is $X$-arc-transitive of prime valency $p$, the induced group $X_\a^{\Ga(\a)}\cong X_\a/X_\a^{[1]}$ is primitive of degree $p$. Hence $N_\a^{\Ga(\a)}\unlhd X_\a^{\Ga(\a)}$ is either trivial or transitive. Suppose that $N_\a^{\Ga(\a)}=1$. Then $N_\a$ fixes every vertex in $\Ga(\a)$. For $\b\in\Ga(\a)$, we have $N_\a\leqslant N_\b$. Since $N$ is transitive on $V(\Ga)$, the stabilizers $N_\a$ and $N_\b$ have the same order, and hence $N_\a=N_\b$. The connectivity of $\Ga$ implies that all vertex stabilizers in $N$ are equal. Thus $N_\a$ is contained in the kernel of the action of $N$ on $V(\Ga)$. Since $N\leqslant X\leqslant\Aut(\Ga)$ acts faithfully, this kernel is trivial, contradicting $N_\a\neq1$. Thus $N_\a^{\Ga(\a)}$ is transitive. Since its degree is the prime $p$, it is primitive. Thus $\Ga$ is $N$-locally primitive and $N$-arc-transitive.

To complete the proof, it suffices to show that $H$ is neither of almost simple type nor of simple diagonal action type. So in what follows, we prove the following two claims.

\textbf{Claim 1.} $H$ is not of almost simple type.

Assume that $H$ is of almost simple type. Since $X\leqslant H \wr\Sy_n$, the argument of \cite[Section 2]{P1993} shows that for $\a\in V(\Ga)$, $N_\a$ is a subdirect subgroup of $R_1\times\cdots\times R_n$ where $1\neq R_i<L_i$ and $R_i\cong R_1$ for all $1\leqslant i\leqslant n$, and so, for every $i$, there exist a projection denoted by \begin{equation*}
\pi_i : R_1\times \cdots\times R_n\to R_i, \quad (x_1, x_2, \ldots, x_n)\mapsto x_i, \text{ where } x_j\in R_j \text{ for } 1\leqslant j\leqslant n,
\end{equation*}
such that $\eta_i:=\pi_i|_{N_\a}$ is a surjective homomorphism. Let $N_i=L_1\times \ldots\times L_{i-1}\times L_{i+1}\times \ldots\times L_{n}$ for $i\in\{1, 2, \ldots, n\}$. Then $N_i$ is a maximal normal subgroup of $N$ and $\ker(\eta_i)=N_i\cap N_\a=(N_i)_\a$.

We assert that each $\eta_i$ is injective, and so $N_\a$ is a full diagonal subgroup of $R_1\times \cdots\times R_n$, i.e., $N_\a\cong R_i$ for all $i$. 
On the contrary, we assume that there exist $i_0\in\{1, 2, \ldots, n\}$ such that $\eta_{i_0}$ is not injective. Then $\eta_{i_0}$ has nontrivial kernel $\ker(\eta_{i_0})=(N_{i_0})_\a$, that is $N_{i_0}$ is not semiregular. 
In view of $G$ is a nonabelian characteristically simple group, one yields that $G$ has no subgroups of index 2 and so $\Ga$ is not bipartite. 
Since $\Ga$ is connected and $N$-locally primitive, Lemma \ref{NQ} implies that $N_{i_0}$ has at most two orbits. If it had exactly two orbits, those orbits would form an $N$-invariant bipartition of $\Ga$, a contradiction. Thus $N_{i_0}$ is transitive. Hence $N=N_{i_0}N_\a$, and
$L_{i_0}\cong N/N_{i_0}\cong N_\a/(N_{i_0})_\a\cong R_{i_0}$, contrary to $R_{i_0}<L_{i_0}$. Thus every $\eta_i$ is injective and $N_\a$ is a full diagonal subgroup of $R_1\times\cdots\times R_n$. In particular, $N_\a\cong R_i$ for every $i$.

By Lemma \ref{subgroup}, $G\leqslant N$. Note the $N_\a\cong R_i$ for $1\leqslant i\leqslant n$. The regularity of $G$ implies that $N=GN_\a$ and $G\cap N_\a=1$, from which it follows that $|N|=|G||N_\a|=|G||R_1|$. Since $N_\a$ is transitive on $\Ga(\a)$ and $|\Ga(\a)|=p$ is a prime, we conclude from Lemma \ref{power} that $v_p(|N_\a|)=1$, and so $v_p(|R_1|)=1$. We assert that $G$ does not contain every minimal normal subgroup of $N$. 
On the contrary, we suppose that $G$ contains a minimal normal subgroup of $N$ say $L_{i_0}$ for some $i_0\in \{1, 2, \ldots, n\}$, that is $L_{i_0}\unlhd G\leqslant N$. 
Further, since $L_{i_0}$ is a nonabelian simple group and $G=\prod_{i=1}^kT_i$ for $T_i\cong T$ nonabelian simple group and $k\geqslant2$ integer, applying the Lemma \ref{obs}, we conclude that $L_{i_0}\cong T$ and then $|R_1|=|N_\a|=|N|/|G|=|L_{i_0}|^n/|T|^k=|L_{i_0}|^{n-k}$, which, together with the fact that $v_p(|R_1|)=1$, implies that  $n-k=1$, and so $|L_{i_0}|=|R_1|$, a contradiction arisen as $R_1$ is a proper subgroup of $L_1\cong L_{i_0}$. Hence the assertion holds.

For $i\in\{1, 2, \ldots, n\}$, let the projection $\varphi_i: N\to L_i$ be defined by $(x_1, x_2, \ldots, x_n)\mapsto x_i$ for $x_j\in L_j$ and $1\leqslant j\leqslant n$. Note that $G=\prod_{i=1}^kT_i$ for $T_i\cong T$ nonabelian simple group and $k\geqslant2$ integer and $N=\prod_{i=1}^nL_i$ for $L_i\cong L$ nonabelian simple group and $n\geqslant2$ integer. Employing the result of \cite[Theorem 1.3 (3)]{DS2020} and the last sentence of the first paragraph and the first sentence of the second paragraph of Case 1 on page 154 of reference \cite{DS2020}, we can draw that one of the following holds. Recall that $T_i\leqslant G\leqslant N$ and the notation $\mathcal{P}(*)$ is given in Definition \ref{defi}.

\begin{itemize}
\item [(a)] $k=n$ and $T_i<L_i$ for every $i$;
\item [(b)] The $T_i$ are pairwise disjoint strips in $N$, $|\mathcal{P}(T_i)|\in\{1, 2\}$ for every $i$, and     $|\mathcal{P}(T_i)|=2$ for at least one $i$;  
\item [(c)] $N\cong\A_m^n$, $N_\a=\A_{m-1}^n$ for $|G|=m^n$ and $m\geqslant 10$. 
\end{itemize}
Assume that case (a) or (c) holds. Derived from $|N|=|G||N_\a|$ that $|N_\a|=|N|/|G|=(|L|/|T|)^n$. Since $\Ga$ is $N$-symmetric, we conclude that $v_p(|N_\a|)=1$, and so $n\,v_p(|L|/|T|)=1$ or $n\,v_p(|\A_{m-1}|)=1$ respectively. It forces that $n=1$, a contradiction as $n\geqslant2$. 

Now that case $(b)$ holds. Noting that $\mathcal{P}(T_i)\neq\emptyset$, one yields that there exists $j$ such that $\varphi_j|_{T_i}(T_i)\neq1$ and $\varphi_j|_{T_i}(T_i)\leqslant L_j$, where $i, j\in\{1, 2, \ldots, n\}$. 
On the other hand, since $T_i$ is nonabelian simple, one yields that the projection $\varphi_j|_{T_i}$ is faithful, that is $T_i\cong\varphi_j|_{T_i}(T_i)\leqslant L_j$ and $|T_i|\leqslant |L_j|$. 
Noting that $|N_\a|=|N|/|G|=|L|^n/|T|^k$ and $v_p(|N_\a|)=1$, one yields that $v_p(|L|^n/|T|^k)=1$ with $n, k\geqslant2$ integers and so $p\mid|L|$ and $p\mid |T|$. 

Let $v_p(|T|)=t_1$ and $v_p(|L|)=t_2$ for positive integers $t_1$ and $t_2$. Then $t_1\leqslant t_2$. By the definition of disjoint of two strips (see Definition \ref{defi}), we yields that $\mathcal{P}(T_i)\cap\mathcal{P}(T_j)=\emptyset$ for all $i\neq j$ with $|\mathcal{P}(T_i)|\in\{1, 2\}$ for all $i$ and $|\mathcal{P}(T_i)|=2$ for at least one $i$, and further, $n>k$ and $nt_2-kt_1=1$ as $v_p(|N_\a|)=1$. 
If $t_1<t_2$, then $nt_2>kt_2>kt_1$, implying $nt_2>kt_1+1$, a contradiction. Thus $t_1=t_2=1$ and $n=k+1$. And further, there exists unique $i_0$ such that $|\mathcal{P}(T_{i_0})|=2$ and $|\mathcal{P}(T_j)|=1$ for all $j\neq i_0$. Since $N_\a\cong R_i<L_i$ for every $1\leqslant i\leqslant n$, one yields $|L|^{k+1}=|T|^k|N_\a|$ divides $|T|^k|L|$ and $|L|^k$ divides $|T|^k$. It follows that $|L|$ divides $|T|$ and $|L|=|T|$. Then $|L|=|N_\a|=|R_i|$, which is impossible as $R_i$ is a proper subgroup of $L_i\cong L$ for every $1\leqslant i\leqslant n$. Hence the Claim 1 holds. 

\textbf{Claim 2.} $H$ is not of simple diagonal action type.

Suppose that $H$ is of simple diagonal action type. The first paragraph of the proof shows that $N_\a^{\Ga(\a)}$ is transitive on $\Ga(\a)$, so $p\mid|N_\a|$. By \cite[Section 3, pp. 303--304]{BP2003}, the group $N_\a$ is a direct product of at least two isomorphic nontrivial groups. Hence $v_p(|N_\a|)\geqslant2$, and therefore $v_p(|X_\a|)\geqslant2$, contrary to Lemma \ref{power}. Hence $H$ is not of simple diagonal action type and so the Claim 2 holds. 

In summary, we have proven the correctness of the Lemma \ref{product action}. \qed

\begin{lemma}\label{twisted wreath action}   
If $X$ is of twisted wreath action type, then $\Ga=\Cay(G, S)$, where $G=\soc(X)$, $s\leqslant2$ and $S$ consists of involutions. 
\end{lemma}

\pf\, Suppose that $X$ is of twisted wreath action type. Then $\soc(X)$ is nonabelian and regular on $V(\Ga)$. Since $G$ is a regular nonabelian characteristically simple subgroup of $X$, Lemma \ref{subgroup} gives $G\leqslant\soc(X)$. Both $G$ and $\soc(X)$ are regular on $V(\Ga)$, so $G=\soc(X)$. Hence $\Ga=\Cay(G,S)$ and $G\unlhd X$. Let $1$ be the identity vertex of $G$. Since $G\unlhd X$ and $G$ is regular, we have $X=G\rtimes X_1$ and $X_1\leqslant\Aut(G,S)$. The connectedness of $\Ga$ gives $\langle S\rangle=G$. Hence the action of $X_1$ on $S$ is faithful, because an element of $X_1\leqslant\Aut(G)$ fixing every element of $S$ fixes every element of $G$. Since $\Ga$ is undirected, $S=S^{-1}$. Also, $|S|=p$ is odd. Every element of $S$ which is not an involution occurs together with its distinct inverse, so $S$ contains at least one involution. Let $I=\{x\in S\mid x^2=1\}$. Since $X_1\leqslant\Aut(G,S)$, the set $I$ is $X_1$-invariant. Since $\Ga$ is $X$-arc-transitive, $X_1$ is transitive on $S$. Hence $I\neq\emptyset$ implies $I=S$, and every element of $S$ is an involution. Finally, $G=\soc(X)\unlhd X$ is regular on $V(\Ga)$ and $\Ga$ has valency $p\geqslant3$. By \cite[Proposition 2.3]{L2001}, $\Ga$ is not $(X,3)$-arc-transitive. Since $\Ga$ is $(X,s)$-transitive, we have $s\leqslant2$. This completes the proof. \qed

We now combine the preceding lemmas to prove Theorem \ref{thm:quasiprimitive}.

\begin{proof}[Proof of Theorem \ref{thm:quasiprimitive}]
Let $\Ga=\Cay(G,S)$ be connected and $p$-valent, where $G=T_1\times\cdots\times T_k\cong T^k$, $k\geqslant2$, each $T_i\cong T$ is simple, and $p$ is an odd prime. Suppose that $G<X\leqslant\Aut(\Ga)$, $\Ga$ is $X$-arc-transitive, and $X$ is quasiprimitive on $V(\Ga)$. By \cite[Theorem 1]{P1993}, $X$ is of affine, almost simple, simple diagonal, product action, or twisted wreath action type.

If $X$ is of affine type, then Lemma \ref{affine} gives $G=\soc(X)\cong\ZZ_2^k$, and so Theorem \ref{thm:quasiprimitive} (1) holds. If $X$ is of almost simple type, then Lemma \ref{almost simple} gives $\Ga\cong\K_{2^k}$, and so Theorem \ref{thm:quasiprimitive} (2) holds. The simple diagonal and product action cases are excluded by Lemmas \ref{simple diagonal} and \ref{product action}, respectively. It remains to consider the twisted wreath action case. By Lemma \ref{twisted wreath action}, $\Ga=\Cay(\soc(X),S)$ and every element of $S$ is an involution. Thus Theorem \ref{thm:quasiprimitive} (3) holds, and the proof is complete. 
\end{proof}

We now turn to the non-quasiprimitive case and consider the normal quotient with respect to a maximal intransitive normal subgroup of $X$.

\section{The non-quasiprimitive case}\label{sec4}

In this Section, we always let $\Ga=\Cay(G,S)$ satisfy Hypothesis \ref{hypothesis}, and assume that $X$ is not quasiprimitive on $V(\Ga)$. Thus $\Ga$ is a connected $p$-valent $X$-arc-transitive Cayley graph of $G=T_1\times\cdots\times T_k\cong T^k$, where $k\geqslant2$, each $T_i\cong T$ is simple, $p$ is an odd prime, and $G<X\leqslant\Aut(\Ga)$. Let $N$ be a maximal intransitive normal subgroup of $X$, $\overline{X}=X/N$ and $\overline{G}=GN/N\cong T^\ell$ for some $1\leqslant\ell\leqslant k$. Now let $L=\soc(\overline{X})=Q_1\times\cdots\times Q_r\cong Q^r$ for some $r\geqslant1$, where $Q$ is simple. The non-quasiprimitive case is described by the following theorem.

\begin{theorem}\label{thm:non-quasiprimitive}
Let $\Ga$ be a graph satisfying Hypothesis \ref{hypothesis}, and suppose that $X$ is not quasiprimitive on $V(\Ga)$. Then $\Ga$ is an $N$-cover of $\Ga_N$ if $\Ga_N\ncong\K_2$. Moreover, one of the following holds.
\begin{enumerate}[leftmargin=1.2em]
\item $\Ga_N$ is a Cayley graph of $\overline{G}\cong\ZZ_2^\ell$.
\item $\overline{X}$ is of affine or twisted wreath action type, and $\Ga_N$ is a Cayley graph of $\soc(\overline{X})$. In the affine case, $\Ga_N\cong\K_8$ and $\ell=1$.
\item $\overline{X}$ is of almost simple type, and either $\Ga_N$ is a complete graph, or $\overline{G}=T$ and either $\overline{G}\unlhd\overline{X}$ or the pair $(\soc(\overline{X}), T)$ is listed in Tables \ref{tab:simple-nonclassical} and \ref{tab:simple-classical-nonsolv}.
\item $\overline{X}$ is of product action type, and either $\overline{G}\unlhd\overline{X}$ or $(Q,T,r)$ is listed in Table \ref{tab:PA-final}.
\end{enumerate}
\end{theorem}

We proceed to establish some technical lemmas.

\begin{lemma}\label{lem:quotient-reduction}
Under the hypotheses of Theorem \ref{thm:non-quasiprimitive}, the group $\overline{X}$ is a quasiprimitive permutation group on $V(\Ga_N)$, and $\overline{G}\cong T^\ell$ for some $1\leqslant\ell\leqslant k$. Moreover, $\overline{G}$ is transitive on $V(\Ga_N)$, and one of the following holds.
\begin{enumerate}[label=(\roman*),leftmargin=2em]
\item $G$ is nonabelian. Then $|V(\Ga_N)|\geqslant3$, $N$ is semiregular on $V(\Ga)$, $\Ga$ is an $N$-cover of $\Ga_N$, and $\Ga_N$ is $\overline{X}$-arc-transitive.
\item $G$ is abelian. Then either $\Ga_N\cong\K_2$, or $\overline{G}\cong\ZZ_2^\ell$ for some $\ell\geqslant2$ and $\Ga_N$ is a Cayley graph of $\overline{G}$. In the latter case, $\Ga$ is an $N$-cover of $\Ga_N$.
\end{enumerate}
\end{lemma}

\pf\, Let $K$ be the kernel of the action of $X$ on the set of $N$-orbits. Then $N\leqslant K\unlhd X$. Since $K$ is intransitive, the maximality of $N$ gives $K=N$. Hence $\overline{X}=X/N$ is a quasiprimitive permutation group on $V(\Ga_N)$. Since $G$ is transitive on $V(\Ga)$, the group $\overline{G}=GN/N\cong G/(G\cap N)$ is transitive on $V(\Ga_N)$. If $T$ is nonabelian, Lemma \ref{obs} gives $G\cap N\cong T^{k-\ell}$ and hence $\overline{G}\cong T^\ell$ for some $1\leqslant\ell\leqslant k$. If $T\cong\ZZ_t$ for a prime $t$, then $G$ is an elementary abelian $t$-group and every quotient of $G$ is elementary abelian, so again $\overline{G}\cong T^\ell$ for some $1\leqslant\ell\leqslant k$.

Suppose first that $G$ is nonabelian. Since $N$ is intransitive, $|V(\Ga_N)|\geqslant2$. If $|V(\Ga_N)|=2$, then the action of $G$ on the two $N$-orbits yields a subgroup of index $2$ in $G$, contrary to the result of Lemma \ref{obs}. Thus $|V(\Ga_N)|\geqslant3$. Since $\Ga$ is $p$-valent and $X$-arc-transitive, it is $X$-vertex-transitive and locally primitive. By Lemma \ref{NQ}, $N$ is semiregular on $V(\Ga)$, $\Ga$ is an $N$-cover of $\Ga_N$, and $\Ga_N$ is $\overline{X}$-arc-transitive. Hence (i) holds.

Now suppose that $G$ is abelian. Then $T\cong\ZZ_t$  and  $\overline{G}\cong\ZZ_t^\ell$ for some prime $t$ and positive integer $\ell$. Since $\overline{G}$ is abelian and transitive on $V(\Ga_N)$, its vertex stabilizers are normal in $\overline{G}$. As $\overline{G}\leqslant\overline{X}$ acts faithfully on $V(\Ga_N)$, these stabilizers are trivial. Hence $\overline{G}$ is regular and $\Ga_N$ is a Cayley graph of $\overline{G}$.

If $\ell=1$, then $\Ga_N$ is a connected Cayley graph of $\ZZ_t$. Its valency is either $1$ or $p$. If its valency is $1$, then $\Ga_N\cong\K_2$. If its valency is $p$, then $t$ cannot be odd because every undirected Cayley graph of a group of odd order has even valency. Hence $t=2$, which again gives $\Ga_N\cong\K_2$.

Assume that $\ell\geqslant2$. Then $|V(\Ga_N)|=t^\ell\geqslant4$. By Lemma \ref{NQ}, $N$ is semiregular on $V(\Ga)$, $\Ga$ is an $N$-cover of $\Ga_N$, and $\Ga_N$ is a $p$-valent $\overline{X}$-arc-transitive Cayley graph of $\overline{G}$. Since $\Ga_N$ is undirected and has odd valency, its connection set contains an involution. Thus $t=2$, and so $\overline{G}\cong\ZZ_2^\ell$. This proves (ii). 

In summary, we have proven the correctness of the Lemma \ref{lem:quotient-reduction}.\qed

The case where $G$ is abelian is now settled. To prove the Theorem , we only need to deal with nonabelian cases. Thus, by Lemma \ref{lem:quotient-reduction}, $\overline{G}\cong T^\ell$ for some $1\leqslant\ell\leqslant k$, $|V(\Ga_N)|\geqslant3$, $\Ga$ is an $N$-cover of $\Ga_N$, and $\Ga_N$ is $\overline{X}$-arc-transitive.

\begin{lemma}\label{lem:simple-quotient}
Suppose that $\overline{G}=T$. Then one of the following holds.
\begin{enumerate}[label=(\roman*),leftmargin=2em]
\item $\Ga_N$ is a complete graph.
\item $\overline{X}$ is of almost simple type, and $\overline{G}=\soc(\overline{X})\unlhd\overline{X}$.
\item $\overline{X}$ is of almost simple type, $T<\soc(\overline{X})$, and if $\Ga_N$ is not complete, then $\Ga_N$ is $\soc(\overline{X})$-arc-transitive, and the pair $(\soc(\overline{X}), T)$ is listed in Tables \ref{tab:simple-nonclassical} and \ref{tab:simple-classical-nonsolv}.
\end{enumerate}
\end{lemma}

\pf\, By Lemma \ref{lem:quotient-reduction}, $\overline{X}$ is quasiprimitive on $V(\Ga_N)$, $\overline{G}=T$ is a transitive nonabelian simple subgroup of $\overline{X}$, and $\Ga_N$ is a connected $p$-valent $\overline{X}$-arc-transitive graph. Hence the hypotheses of \cite[Theorem 1.2]{L2025} are satisfied. It follows that either $\overline{X}\cong\AGL_3(2)$ and $\Ga_N\cong\K_8$, or $\overline{X}$ is of almost simple type.
The first possibility is contained in (i). Hence we may assume that $\overline{X}$ is of almost simple type and let $L=\soc(\overline{X})$. By Lemma \ref{subgroup}, we have $T\leqslant L$. If $T=L$, then $\overline{G}=T\unlhd\overline{X}$, and (ii) holds. Thus assume that $T<L$. If $\Ga_N$ is complete, then (i) holds. For the remainder of the proof, we assume that $\Ga_N$ is not complete.

Since $T$ is transitive on $V(\Ga_N)$, we have $L=TL_{\overline{\a}}$, and so $L_{\overline{\a}}\neq1$ as $T<L$. Since $L\unlhd\overline{X}$, we have $L_{\overline{\a}}\unlhd(\overline{X})_{\overline{\a}}$. Hence $L_{\overline{\a}}^{\Ga_N(\overline{\a})}$ is normal in the primitive permutation group $(\overline{X})_{\overline{\a}}^{\Ga_N(\overline{\a})}$.
We claim that $L_{\overline{\a}}^{\Ga_N(\overline{\a})}\neq1$. Otherwise, $L_{\overline{\a}}$ fixes every vertex in $\Ga_N(\overline{\a})$ and so $L_{\overline{\a}}\leqslant L_{\overline{\b}}$ for some $\overline{\b}\in\Ga_N(\overline{\a})$. Since $L$ is vertex-transitive, $|L_{\overline{\a}}|=|L_{\overline{\b}}|$, and so $L_{\overline{\a}}=L_{\overline{\b}}$. By the connectivity of $\Ga_N$, all vertex stabilizers in $L$ are equal. Thus $L_{\overline{\a}}$ lies in the kernel of the action of $L$ on $V(\Ga_N)$. Since $L\leqslant\overline{X}$ and $\overline{X}$ acts faithfully on $V(\Ga_N)$, the action of $L$ is faithful, and so this kernel is trivial. Thus $L_{\overline{\a}}=1$, a contradiction. Then $L_{\overline{\a}}^{\Ga_N(\overline{\a})}$ is transitive on $\Ga_N(\overline{\a})$, and so $\Ga_N$ is $L$-arc-transitive.

Since $T$ and $L$ are nonabelian simple and $T<L$, the subgroup $T$ is core-free in $L$. We now distinguish two cases.

\textbf{Claim 1.} If $L_{\overline{\a}}$ is solvable, then $(L,T,L_{\overline{\a}})$ occurs in rows 1--4 of Table \ref{tab:simple-nonclassical}.

Suppose that $L_{\overline{\a}}$ is solvable. Applying the \cite[Section 2]{DM1980} and \cite[Theorem 2.1]{GHX2015}, we conclude that either $p\geqslant5$ and $L_{\overline{\a}}\leqslant\AGL_1(p)\times\ZZ_{p-1}$, or $p=3$ and $L_{\overline{\a}}\in\{\ZZ_3,\,\Sy_3,\,\Sy_3\times\ZZ_2,\,\Sy_4,\,\Sy_4\times\ZZ_2\}$. By \cite[Proposition 4.2]{LX2022}, $L$ is not an exceptional simple group of Lie type. 

(I). Assume that $L=\A_n$ acting naturally on set $\Omega:=\{1, 2, \ldots, n\}$. Now, we apply \cite[Proposition 4.3]{LX2022} to the factorization $L=TL_{\overline{\a}}$. First, we suppose that \cite[Proposition 4.3 (a)]{LX2022} holds, that is $L_{\overline{\a}}$ is transitive on $\Omega$ and $\A_{n-1}\leqslant T\leqslant\Sy_{n-1}$. Noting that $T$ is nonabelian simple, we have $T=\A_{n-1}$ and $n\geqslant 6$. This gives row 1 of Table \ref{tab:simple-nonclassical}.

\begin{table}[htbp]
\centering
\renewcommand{\arraystretch}{1.3}
\setlength{\tabcolsep}{7pt}
\caption{The non-complete nonclassical cases.}
\label{tab:simple-nonclassical}
{\scriptsize
\begin{tabular}{|c|c|c|c|p{6cm}|}
\hline
row & $L$ & $T$ & $L_{\overline{\a}}$ & remark\\
\hline
1 & $\A_n$ & $\A_{n-1}$ & transitive on $\Omega$ & $|\Omega|=n\geqslant6$\\
\hline
2 & $\M_{23}$ & $\M_{22}$ & $\ZZ_{23}$ & $p=23$, $\Ga_N$ is a Cayley graph of $\M_{22}$\\
\hline
3 & $\PSL_2(p)$ & $\A_5$ & $\leqslant\ZZ_p\rtimes\ZZ_{(p-1)/2}$ & $p\in\{11,19,29\}$\\
\hline
4 & $\M_{24}$ & $\M_{23}$ & $\PSL_3(2)$ & $p=7$, $\Ga_N$ is unique up to isomorphism\\
\hline
\end{tabular}}
\end{table}

Suppose next that \cite[Proposition 4.3(b)]{LX2022} holds, that is $n=r^f$ for some prime $r$, the group $L_{\overline{\a}}$ is $2$-homogeneous on $\Omega$, and $\A_{n-2}\leqslant T\leqslant \Sy_{n-2}\times \Sy_2$. Since $T$ is nonabelian simple, we have $T=\A_{n-2}$ and $n\geqslant7$. Hence $|L:T|=n(n-1)$ divides $|L_{\overline{\a}}|$.
If $p=3$, then $n(n-1)$ divides $|L_{\overline{\a}}|$ which is less than or equal to 48. Due to $n$ being greater than or equal to 7 and $|L_{\overline{\a}}|\neq 42$, a contradiction arises. Hence $p\geqslant5$ and $L_{\overline{\a}}\leqslant \AGL_1(p)\times \ZZ_{p-1}$. Note that $p\mid |L_{\overline{\a}}|$. The group $\AGL_1(p)\times \ZZ_{p-1}$ has a unique Sylow $p$-subgroup $P\cong\ZZ_p$. Hence $P\leqslant L_{\overline{\a}}$ and $P\unlhd L_{\overline{\a}}$. Since $L_{\overline{\a}}$ is $2$-homogeneous on $\Omega$, it is primitive on $\Omega$, and so every nontrivial normal subgroup of $L_{\overline{\a}}$ is transitive on $\Omega$. Thus $P$ is transitive on $\Omega$, and so $n\mid p$. Together with $p\leqslant n$, this gives $n=p$. Hence $L\cong\A_p$ and $T\cong\A_{p-2}$. Since $P\unlhd L_{\overline{\a}}\leqslant L$, we have $L_{\overline{\a}}\leqslant \N_L(P)$. 
The group $\A_p$ contains $(p-1)!$ $p$-cycles, and each Sylow $p$-subgroup contains exactly $p-1$ nonidentity elements. Hence $\A_p$ has $(p-2)!$ Sylow $p$-subgroups. Since the number of Sylow $p$-subgroups is $|\A_p:\N_{\A_p}(P)|$, we obtain $|\N_{\A_p}(P)|=p(p-1)/2$. On the other hand, $p(p-1)=|L:T|$ divides $|L_{\overline{\a}}|$, contradicting $L_{\overline{\a}}\leqslant \N_L(P)$. Hence this case does not occur.

Now consider the case (c) of \cite[Proposition 4.3]{LX2022}, that is, $(L_{\overline{\a}},n)=(\AGL_1(8),8)$, $(\AGammaL_1(8),8)$ or $(\AGammaL_1(32),32)$. Recall that $\Ga_N$ is $p$-valent $L$-arc-transitive. By \cite[Theorem 2.1]{GHX2015}, we get that $p$ is the largest prime divisor of $|L_{\overline{\a}}|$ and $v_p(|L_{\overline{\a}}|)=1$. For $(L_{\overline{\a}},n)=(\AGL_1(8),8)$ or $(\AGammaL_1(8),8)$, we have $|L_{\overline{\a}}|\in\{56,168\}$ and so $p=7$. Due to $L_{\overline{\a}}\leqslant\AGL_1(p)\times\ZZ_{p-1}$, one yields that $|L_{\overline{\a}}|\mid 7\cdot 6^2$, a contradiction arisen as $v_2(|L_{\overline{\a}}|)=3>2=v_2(7\cdot 6^2)$. If $(L_{\overline{\a}},n)=(\AGammaL_1(32),32)$, then $p=31$ and $|L_{\overline{\a}}|\mid 31\cdot30^2$. This is impossible because $v_2(|L_{\overline{\a}}|)=5> 2=v_2(31\cdot30^2)$. Thus case (c) does not occur.

Suppose that (d) and (e) of \cite[Proposition 4.3]{LX2022} hold, that is $(L,T)=(\A_6,\A_5)$ and either $L_{\overline{\a}}\leqslant \Sy_4\times \Sy_2$ or $L_{\overline{\a}}\leqslant \Sy_3\wr\Sy_2$. Recall that $\Ga_N$ is $p$-valent $L$-arc-transitive. By Lemma \ref{power}, we have $v_p(|L_{\overline{\a}}|)=1$. Thus $p=3$. 
Since $T=\A_5$ is transitive on $V(\Ga_N)$, we have $|V(\Ga_N)|$ divides 60, and so \cite{CP2026} gives $\Ga_N\cong\F30A$ or $\F60A$, where $\F30A$ is the bipartite. Since $L\cong\A_6$ is perfect and vertex-transitive on $\Ga_N$, the graph $\Ga_N$ is nonbipartite. Thus $\Ga_N\ncong\F30A$, and so $\Ga_N\cong\F60A$. Hence $L_{\overline{\a}}\cong\Sy_3$ and $T$ is regular on $V(\Ga_N)$. Hence $L\cong\A_6$ would admit a factorization $\A_6=\A_5\Sy_3$. The following Magma computation (Listing \ref{lst:A6-factorization}) shows that no such pair exists. Thus cases (d) and (e) do not occur.

Finally, suppose that \cite[Proposition 4.3(f)]{LX2022} holds, that is, $(L,L_{\overline{\a}},T)$ is one of the triples listed in \cite[Table 4.2]{LX2022}. However, in every row of that table, either $L$ or $T$ is not nonabelian simple. Hence case (f) does not occur.

(II). Suppose now that $L$ is a sporadic simple group. Applying the \cite[Proposition 4.4]{LX2022}, we conclude that one of the following holds.
\begin{enumerate}[leftmargin=1.2em,label=(\arabic*)]
\item $L=\M_{12}$, $T=\M_{11}$, and $L_{\overline{\a}}$ is transitive on $[L:T]$.
\item $L=\M_{24}$, $T=\M_{23}$, and $L_{\overline{\a}}$ is transitive on $[L:T]$.
\item $(L,T,L_{\overline{\a}})$ occurs in \cite[Table 4.3]{LX2022}.
\end{enumerate}

Assume first that (1) holds. Since $\Ga_N$ is $L$-arc-transitive of valency $p$, we have $p\mid |L_{\overline{\a}}|$. As $L_{\overline{\a}}\leqslant L=\M_{12}$ and $p$ is odd, it follows from $|\M_{12}|=2^6\cdot3^3\cdot5\cdot11$ that $p\in\{3,5,11\}$. Since $L=TL_{\overline{\a}}$, we have $12=|L:T|$ divides $|L_{\overline{\a}}|$. Note that $L_{\overline{\a}}\leqslant\AGL_1(p)\times\ZZ_{p-1}$. If $p=5$ or $11$, then $|L_{\overline{\a}}|\mid p(p-1)^2$, but $12\nmid p(p-1)^2$, a contradiction. Hence $p=3$. The transitivity of $T\cong\M_{11}$ on $V(\Ga_N)$ gives $|V(\Ga_N)|\mid|\M_{11}|=7920$. By \cite{CP2026}, the only connected cubic symmetric graph of order dividing $7920$ admitting $\M_{12}$ as an arc-transitive group is $\F7920C$. Thus $\Ga_N\cong\F7920C$ and $|V(\Ga_N)|=7920$. The equality $|V(\Ga_N)|=|T|$, together with the transitivity of $T$ on $V(\Ga_N)$, shows that $T$ is regular on $V(\Ga_N)$. Hence $\Ga_N$ is a connected cubic symmetric Cayley graph of $T\cong\M_{11}$. By \cite[Theorem 1.1]{XFWX2005}, such a graph is normal unless $T\cong\A_{47}$. Thus $T\unlhd\Aut(\Ga_N)$, and in particular $T\unlhd L$. This contradicts the simplicity of $L\cong\M_{12}$ and the proper inclusion $1<T<L$. Hence case (1) does not occur.

Suppose that (2) holds. Since $\Ga_N$ is $L$-arc-transitive of valency $p$, we have $p\mid |L_{\overline{\a}}|$. As $L_{\overline{\a}}\leqslant L=\M_{24}$ and $p$ is odd, it follows from $|\M_{24}|=2^{10}\cdot3^3\cdot5\cdot7\cdot11\cdot23$ that $p\in\{3,5,7,11,23\}$. Moreover, since $L_{\overline{\a}}$ is transitive on $[L:T]$ and $|L:T|=24$, we have $24\mid |L_{\overline{\a}}|$.
If $p\neq3$, then $|L_{\overline{\a}}|\mid p(p-1)^2$, but $24\nmid p(p-1)^2$, a contradiction. Hence $p=3$. Then $L_{\overline{\a}}\cong\Sy_4$ or $\Sy_4\times\ZZ_2$. Assume that $L_{\overline{\a}}\cong\Sy_4$. Then $|T\cap L_{\overline{\a}}|=|T||L_{\overline{\a}}|/|L|=1$ as $T\cong\M_{23}$ and $L\cong\M_{24}$. Hence $T$ is regular on $V(\Ga_N)$, and $\Ga_N$ is a connected cubic Cayley graph of $T$, say $\Ga_N=\Cay(T,U)$ for some $1\not\in U\subseteq T$. By \cite[Theorem 1.1]{XFWX2005}, the graph $\Ga_N$ is a normal Cayley graph as $T\not\cong\A_{47}$ and $\Aut(\Ga_N)=T\rtimes\Aut(T,U)$. It follows that $\Aut(T,U)\leqslant\Sy_3$, a contradiction arisen as $L_{\overline{\a}}\cong\Sy_4\leqslant\Aut(T,U)$. 
Suppose that $L_{\overline{\a}}\cong\Sy_4\times\ZZ_2$. Then $|L_{\overline{\a}}|=48$ and $|T\cap L_{\overline{\a}}|=|T||L_{\overline{\a}}|/|L|=|\M_{23}|48/|\M_{24}|=2$. It follows from \cite[Section 2]{DM1980} that $\Ga_N$ is $(L,5)$-arc-regular. Let $\Sigma$ be obtained from $\Ga_N$ by placing a new vertex at the midpoint of each edge of $\Ga_N$. Then $\Sigma$ is bipartite, with one part consisting of the vertices of valency 3 and the other part consisting of the vertices of valency 2. Then $\Sigma$ is bipartite with valencies 3 and 2 on its two parts. By \cite[p. 884, Section 2.1]{L2009}, the graph $\Sigma$ is locally $(\M_{24},9)$-arc-transitive. This is impossible by \cite[Table 3]{L2009}, which contains no locally $(\M_{24},9)$-arc-transitive graph. Hence (2) does not occur.

Finally, suppose that (3) holds, that is $(L,T,L_{\overline{\a}})$ occurs in \cite[Table 4.3]{LX2022}. In rows 1, 4, 5, 8 and 10--13 of \cite[Table 4.3]{LX2022}, either $L$ or $T$ is not nonabelian simple, so these rows do not occur. In rows 2 and 3, the possible orders of $L_{\overline{\a}}$ are 72, 144 and 432, and for row 9, we have $|L_{\overline{\a}}|=2000$. Hence no odd prime divisor $p$ of $|L_{\overline{\a}}|$ satisfies $v_p(|L_{\overline{\a}}|)=1$, contrary to Lemma \ref{power}. Thus only rows 6 and 7 remain. In row 7, the simplicity of $T$ excludes $\PSL_3(4):\ZZ_2$ and $\ZZ_2^4:\A_7$, and so in both rows we have $L\cong\M_{23}$ and $T\cong\M_{22}$.
In row 6, $L_{\overline{\a}}\cong\ZZ_{23}$. Since $p\mid|L_{\overline{\a}}|$, we have $p=23$. Moreover, since $23=|L:T|=|L_{\overline{\a}}:T\cap L_{\overline{\a}}|$, it follows that $T\cap L_{\overline{\a}}=1$. Thus $T$ is regular on $V(\Ga_N)$, and $\Ga_N$ is a $23$-valent Cayley graph of $\M_{22}$, as in row 2 of Table \ref{tab:simple-nonclassical}.
Now consider row 7, where $L_{\overline{\a}}\cong\ZZ_{23}\rtimes\ZZ_{11}$. By \cite[Theorem 2.1]{GHX2015}, the valency $p$ is the largest prime divisor of $|L_{\overline{\a}}|$. Then $p=23$. Let $\{\overline{\a},\overline{\b}\}$ be an edge of $\Ga_N$. Since $L$ is arc-transitive, $|L_{\overline{\a}}:L_{\overline{\a}\overline{\b}}|=p$ and so $|L_{\overline{\a}\overline{\b}}|=11$. Since $\Ga_N$ is undirected and $L$-arc-transitive, $L_{\overline{\a}\overline{\b}}$ has index $2$ in $L_{\{\overline{\a},\overline{\b}\}}$. Thus $|L_{\{\overline{\a},\overline{\b}\}}|=22$ and $L_{\{\overline{\a},\overline{\b}\}}\leqslant \N_L(L_{\overline{\a}\overline{\b}})$. 
However, for $L\cong\M_{23}$, \cite{Atlas} gives $\N_L(\ZZ_{11})\cong \ZZ_{11}\rtimes\ZZ_5$, of order 55. This is impossible because $L_{\{\overline{\a},\overline{\b}\}}\leqslant \N_L(L_{\overline{\a}\overline{\b}})$ has order $22$, but $22\nmid55$. Hence row 7 does not occur.

(III) Now suppose that $L$ is a classical simple group of Lie type. Then by \cite[Theorem 1.1]{LX2022}, $(L,T,L_{\overline{\a}})$ occurs in \cite[Tables 1.1 and 1.2]{LX2022}. Since $L$ and $T$ are nonabelian simple and $L=TL_{\overline{\a}}$, row 1 of \cite[Table 1.1]{LX2022} and rows 7--10, 17 and 21 of \cite[Table 1.2]{LX2022} are excluded because $T$ is not simple. 
A direct order calculation shows that $|L:T|\nmid|L_{\overline{\a}}|$ in rows 2 and 6 of \cite[Table 1.2]{LX2022}, 
so these rows are excluded. By \cite[Theorem 2.1]{GHX2015}, we get that $p$ is the largest prime divisor of $|L_{\overline{\a}}|$ and $v_p(|L_{\overline{\a}}|)=1$. 
Inspection of the orders in rows 11--16 and 18--20 of \cite[Table 1.2]{LX2022} shows that no prime divisor $p$ satisfies both $p=\max\Pi(|L_{\overline{\a}}|)$ and $v_p(|L_{\overline{\a}}|)=1$, except for $(L,T,L_{\overline{\a}})=(\PSU_4(3),\PSL_3(4),\ZZ_3^4\rtimes\D_{10})$. This is a direct prime-factorization check of the displayed stabilizer orders. 
This exceptional case and rows 22--28 are excluded as $|L_{\overline{\a}}|\notin\{3,6,12,24,48\}$ for $p=3$, and $|L_{\overline{\a}}|\nmid p(p-1)^2$ for $p\geqslant5$. The remaining finite case with $(L,T)=(\PSL_2(p),\A_5)$ gives row 3 of Table \ref{tab:simple-nonclassical} for $p\in\{11,19,29\}$. The case $(L,T)=(\PSL_2(59),\A_5)$ gives $\Ga_N\cong\K_{60}$ and is excluded by our standing assumption that $\Ga_N$ is not complete. 

Assume that one of rows 2--9 of \cite[Table 1.1]{LX2022} occurs. Since $T$ is simple, the nontrivial normal subgroup displayed in the fourth column of \cite[Table 1.1]{LX2022} is equal to $T$. A direct calculation gives $|L:T|>48$ in rows 2--9 of \cite[Table 1.1]{LX2022} whenever both $L$ and $T$ are nonabelian simple. If $p=3$, then the solvable-stabilizer classification gives $|L_{\overline{\a}}|\leqslant48$, whereas $|L:T|\mid|L_{\overline{\a}}|$. This is impossible. Hence $p\geqslant5$ and $L_{\overline{\a}}\cong (\ZZ_p\rtimes \ZZ_a)\times \ZZ_b$, where $b\mid a\mid p-1$ (see \cite[Theorem 2.1]{GHX2015}). 
In particular, $v_p(|L_{\overline{\a}}|)=1$ and $|L_{\overline{\a}}|\mid p(p-1)^2$. Let $q=t^f$. If $t=p$, then $v_p(|L:T|)\geqslant2$ in each of rows 2--9, contrary to $v_p(|L_{\overline{\a}}|)=1$ and $|L:T|\mid|L_{\overline{\a}}|$. Thus $t\neq p$. Since $L_{\overline{\a}}\cong(\ZZ_p\rtimes\ZZ_a)\times\ZZ_b$ and $\ZZ_a$ acts faithfully on $\ZZ_p$, the first factor has no nontrivial normal $t$-subgroup. Hence $\O_t(L_{\overline{\a}})\leqslant\ZZ_b$, and so $\O_t(L_{\overline{\a}})$ is cyclic.

Suppose that row 2 of \cite[Table 1.1]{LX2022} occurs, that is, $L\cong\PSL_4(q)$, $T\cong\PSp_4(q)$ and $L_{\overline{\a}}\leqslant \E_q^3:\ZZ_{(q^3-1)/d}.\ZZ_3$, where $d=(4,q-1)$. The inclusion $L_{\overline{\a}}\cap \E_q^3\leqslant\O_t(L_{\overline{\a}})$, together with the fact that $\E_q^3$ is elementary abelian and $\O_t(L_{\overline{\a}})$ is cyclic, gives $|L_{\overline{\a}}\cap \E_q^3|\leqslant t$. Also, $L_{\overline{\a}}/(L_{\overline{\a}}\cap \E_q^3)$ is isomorphic to a subgroup of $\ZZ_{(q^3-1)/d}.\ZZ_3$. The relation $t\nmid q^3-1$ then gives $v_t(|L_{\overline{\a}}|)\leqslant1+v_t(3)$. Now $v_t(|L:T|)=2f$, and the divisibility $|L:T|\mid|L_{\overline{\a}}|$ gives $2f\leqslant1+v_t(3)$. Thus $t=3$ and $f=1$, so $q=3$. However, $|\PSL_4(3):\PSp_4(3)|=234\nmid1053$, so row 2 does not occur. 
Applying the same $t$-part argument to the stabilizer structures in rows 3, 5, 6, 7 and 8 of \cite[Table 1.1]{LX2022} excludes these rows, while row 9 reduces to $q=2$.
For the remaining case in row 9, $L\cong\POmega_8^+(2)$, $T\cong\Omega_7(2)$ and $L_{\overline{\a}}\leqslant\E_2^6:\ZZ_{15}.\ZZ_4$. Thus $|L_{\overline{\a}}|\mid3840$. Since $p\geqslant5$ and $p\mid|L_{\overline{\a}}|$, we obtain $p=5$. By \cite[Theorem 2.1]{GHX2015}, $|L_{\overline{\a}}|\mid5\cdot4^2=80$, whereas $|L:T|=120\mid|L_{\overline{\a}}|$, which is impossible.

Assume next that row 4 of \cite[Table 1.1]{LX2022} holds, that is, $L\cong\PSp_4(q)$, $T\cong\Sp_2(q^2)$ and $L_{\overline{\a}}\leqslant \E_q^3:\ZZ_{q^2-1}.\ZZ_2$, where $q=2^f$. The subgroup $L_{\overline{\a}}\cap \E_q^3$ is contained in the cyclic group $\O_2(L_{\overline{\a}})$. 
Since $\E_q^3$ is elementary abelian for $q=2^f$, this intersection has order at most 2. Moreover, $L_{\overline{\a}}/(L_{\overline{\a}}\cap \E_q^3)$ is isomorphic to a subgroup of $\ZZ_{q^2-1}.\ZZ_2$. Since $q^2-1$ is odd, $v_2(|L_{\overline{\a}}|)\leqslant2$. Now $v_2(|L:T|)=2f$, and $|L:T|\mid|L_{\overline{\a}}|$ gives $2f\leqslant2$. Hence $q=2$, which is impossible because $\PSp_4(2)$ is not simple.

Thus none of the rows in \cite[Table 1.1]{LX2022} occurs and so Claim 1 follows.

\textbf{Claim 2.} If $L_{\overline{\a}}$ is nonsolvable, then $(L,T)$ is one of the pairs listed in rows 1 and 4 of Table \ref{tab:simple-nonclassical} or in Table \ref{tab:simple-classical-nonsolv}.

Suppose first that $L$ or $T$ is alternating. By \cite[Theorem 6.1]{LLarXiv2021}, the cases for which $\Ga_N\cong\K_{p+1}$ are excluded by the assumption that $\Ga_N$ is not complete and the remaining possibilities are $(L,T,L_{\overline{\a}})=(\A_{12},\A_{11},\PSL_2(11))$ or $L=\A_n$, $T=\A_{n-1}$ with $L_{\overline{\a}}$ being a transitive but not 2-transitive subgroup of $L$. These are included in row 1 of Table \ref{tab:simple-nonclassical}. 
Suppose next that $L$ is a sporadic simple group. By \cite[Theorem 6.2]{LLarXiv2021}, the cases with $L_{\overline{\a}}\cong\M_{11}$ or $\M_{23}$ give $\K_{12}$ or $\K_{24}$ and are excluded. The only remaining possibility is $(L,T,L_{\overline{\a}})=(\M_{24},\M_{23},\PSL_3(2))$ with $p=7$, giving row 4 of Table \ref{tab:simple-nonclassical}. 
By \cite[Lemma 4.1]{LLarXiv2021}, $L$ is not an exceptional simple group of Lie type. 
Finally, if $L$ is a classical simple group, then \cite[Theorems 8.1--8.7]{LLarXiv2021} gives the cases in Table \ref{tab:simple-classical-nonsolv}. This proves Claim 2.

Combining Claims 1 and 2, we obtain Lemma \ref{lem:simple-quotient} (iii), and so the lemma follows. \qed

\begin{table}[htbp]
\centering
\setlength{\tabcolsep}{6pt}
\renewcommand{\arraystretch}{1.3}
\caption{The non-complete classical cases with nonsolvable $L_{\overline{\a}}$.}
\label{tab:simple-classical-nonsolv}
{\scriptsize
\begin{tabular}{|c|c|c|c|}
\hline
$L$ & $T$ & $K\unlhd L_{\overline{\a}}^{\Ga_N(\overline{\a})}$ & remark \\
\hline
$\PSU_6(2)$ & $\PSU_5(2)$ & $\SL_3(2)$ & $p=7$ \\
\hline
$\PSU_{2m}(q)$ & $\PSU_{2m-1}(q)$ & $\SL_2(q^m)$ & $p=q^m+1$, $q=2^f$, $mf$ a power of 2 \\
\hline
$\Omega_{2m+1}(q)$ & $\POmega^-_{2m}(q)$ & $\PSL_d(q^{m/d})$ & $p=\dfrac{q^m-1}{q^{m/d}-1}$, $d$ odd prime, $mf$ a power of $d$ \\
\hline
$\POmega^-_{10}(q)$ & $\PSU_5(q)$ & $\SL_2(q^4)$ & $p=q^4+1$, $q>2$, $f$ a power of 2 \\
\hline
$\POmega^+_8(q)$ & $\PSp_6(q)$ & $\SL_2(q^2)$ & $p=q^2+1$, $q=2^f$, $f$ a power of 2 \\
\hline
$\POmega^+_8(3)$ & $\Omega^+_8(2)$ & $\PSL_3(3)$ & $p=13$, $L_{\overline{\a}}\cong\ZZ_3^6\rtimes\PSL_3(3)$ \\
\hline
$\POmega^+_{12}(2)$ & $\Omega^-_{10}(2),\PSp_{10}(2)$ & & $\SL_5(2)\leqslant L_{\overline{\a}}\leqslant\ZZ_2^{20}\rtimes\SL_5(2)$ \\
\hline
$\POmega^+_{12}(2)$ & $\PSp_{10}(2)$ & $\SL_3(2)$ & \\
\hline
$\POmega^+_{2m}(q)$ & $\PSp_{2m-2}(q)$ & $\SL_2(q^{m/2})$ & $m\geqslant8$, $r=2$, $mf$ a power of $2$ \\
\hline
$\POmega^+_{2m}(q)$ & $\PSp_{2m-2}(q),\Omega^-_{2m-2}(q)$ & $\PSL_d(q^{m/d})$ & $r=2$, $d$ odd prime, $mf$ a power of $d$ \\
\hline
$\POmega^+_{2m}(q)$ & $\Omega_{2m-1}(q)$ & $\PSL_d(q^{m/d})$ & $r,d$ odd, $mf$ a power of $d$ \\
\hline
$\POmega^+_{2m}(q)$ & $\POmega^-_{2m-2}(q)$ & $\PSL_m(q)$ & $m$ odd prime, $f$ a power of $m$ \\
\hline
$\PSp_4(q)$ & $\PSL_2(q^2)$ & $\PSL_2(q^2),\SL_2(q)$ & $q=2^f$, $f>2$ a power of 2 \\
\hline
$\PSp_6(q)$ & $G_2(q),\PSL_2(q^3)$ & $\SL_2(q^2)$ & $q>2$, $f$ a power of 2 \\
\hline
$\PSp_{12}(2)$ & $\Omega^-_{12}(2)$ & $\PSL_3(2),\PSL_5(2)$ & \\
\hline
$\PSp_{2m}(q)$ & $\Omega^-_{2m}(q)$ & $\PSL_d(q^{m/d})$ & $m\geqslant5$, $d$ odd prime, $mf$ a power of $d$ \\
\hline
$\PSp_{2m}(q)$ & $\Omega^-_{2m}(q)$ & $\PSL_2(q^{m/2})$ & $m\geqslant4$, $mf$ a power of 2 \\
\hline
$\PSp_{2m}(q)$ & $\Omega^+_{2m}(q)$ & $\PSL_2(q^m)$ & $m\geqslant4$, $mf$ a power of 2 \\
\hline
$\PSL_n(q)$ & $\PSL_{n-1}(q)$ & $\PSL_d(q^{n/d})$ & $d$ odd prime, $nf$ a power of $d$ \\
\hline
$\PSL_4(4)$ & $\PSp_4(4)$ & $\PSL_3(2)$ & \\
\hline
$\PSL_n(q)$ & $\PSL_m(q^{n/m})$ & $\PSL_{n-1}(q)$ & $m$ a proper divisor of $n$ \\
\hline
$\PSL_n(q)$ & $\PSp_m(q^{n/m})$ & $\PSL_{n-1}(q)$ & $m\geqslant4$ even \\
\hline
$\PSL_n(q)$ & $\POmega^-_m(q^{n/m})$ & $\PSL_{n-1}(q)$ & $4\mid m$, $m\geqslant8$ \\
\hline
$\PSL_n(q)$ & $\G_2(q^{n/6})$ & $\PSL_{n-1}(q)$ & $6\mid n$, $q$ even \\
\hline
$\PSL_6(2)$ & $\PSU_3(3)$ & $\PSL_5(2)$ & \\
\hline
\end{tabular}}
\end{table}

In the above lemma, we described the case of nonabelian simple group, and now we deal with the case of multiple nonabelian simple groups, i.e., $\ell\geqslant2$. Since $\overline{X}$ is quasiprimitive on $V(\Ga_N)$, it is of affine, almost simple, simple diagonal, product action, or twisted wreath action type. We treat these five types separately. In the rest of this Section, we always let $\overline{G}\cong T^\ell$ for some $\ell\geqslant 2$ and $L=\soc(\overline{X})=Q_1\times\cdots\times Q_r\cong Q^r$ for some $r\geqslant 1$.

\begin{lemma}\label{lem:affine-large}
Suppose that $\overline{G}\cong T^\ell$ with $\ell\geqslant2$. Then $\overline{X}$ is not of affine type.
\end{lemma}

\pf\, Assume that $\overline{X}$ is of affine type. The group $\overline{G}\cap\soc(\overline{X})$ is an abelian normal subgroup of $\overline{G}$, so Lemma \ref{obs} gives $\overline{G}\cap\soc(\overline{X})=1$. By \cite[Theorem 1.3]{DS2020} and the first two paragraphs of its proof, $\overline{G}=T_1\times\cdots\times T_\ell$ with $T_i\cong\SL_3(2)$, $|V(\Ga_N)|=8^\ell$ and $\overline{G}_{\overline{\a}}=K_1\times\cdots\times K_\ell$, where $|T_i:K_i|=8$. Thus $|K_i|=21$.

For a finite group $Y$, let $n_7(Y)$ denote the number of Sylow 7-subgroups of $Y$. For each $i\in\{1,\ldots,\ell\}$, let $P_i\in\Syl_7(K_i)$. Noting that $|K_i|=21$ yields $n_7(K_i)\equiv1\pmod7$ and $n_7(K_i)\mid3$. Hence $n_7(K_i)=1$, so $P_i\unlhd K_i$ and $K_i\leqslant\N_{T_i}(P_i)$. On the other hand, $n_7(T_i)\equiv1\pmod7$ and $n_7(T_i)\mid24$. Since $T_i\cong\SL_3(2)$ is simple, $n_7(T_i)\neq1$, and so $n_7(T_i)=8$. Thus $|\N_{T_i}(P_i)|=168/8=21$, giving $K_i=\N_{T_i}(P_i)$. Since $P_i$ is characteristic in $K_i$, we have $\N_{T_i}(K_i)\leqslant\N_{T_i}(P_i)=K_i$, and hence $\N_{T_i}(K_i)=K_i$.
Consider the action of $K_i$ by right multiplication on the right cosets of $K_i$ in $T_i$. For $y_i\in T_i$, the stabilizer of $K_iy_i$ in $K_i$ is $K_i\cap K_i^{y_i}$, so the orbit containing $K_iy_i$ has length $|K_i:K_i\cap K_i^{y_i}|$. This length is $1$ if and only if $K_i=K_i^{y_i}$, or equivalently, $y_i\in\N_{T_i}(K_i)=K_i$. Thus $K_i$ has exactly one fixed point. Every $K_i$-orbit length divides 21, and the remaining seven cosets contain no fixed point. Hence they form a single orbit of length 7.

Let $H=\overline{G}_{\overline{\a}}=K_1\times\cdots\times K_\ell$. For $\overline{\b}\in V(\Ga_N)$, choose $x=(x_1,\ldots,x_\ell)\in\overline{G}$ such that $\overline{\b}=\overline{\a}^x$. Then $\overline{G}_{\overline{\b}}=H^x$, and hence $H_{\overline{\b}}=H\cap H^x$. Since $H^x=K_1^{x_1}\times\cdots\times K_\ell^{x_\ell}$, we have $H\cap H^x=(K_1\cap K_1^{x_1})\times\cdots\times(K_\ell\cap K_\ell^{x_\ell})$. Note that $|\overline{\b}^H|=\prod_{j=1}^{\ell}|K_j:K_j\cap K_j^{x_j}|$. If $\overline{\b}\neq\overline{\a}$, then $x\notin H$, so $x_i\notin K_i$ for some $i$. The preceding paragraph gives $|K_i:K_i\cap K_i^{x_i}|=7$. Hence $7\mid|\overline{\b}^H|$ for every $\overline{\b}\neq\overline{\a}$.

The subgroup $H$ preserves $\Ga_N(\overline{\a})$. Moreover, the preceding paragraph shows that $\overline{\a}$ is the unique fixed point of $H$ on $V(\Ga_N)$. Since $\overline{\a}\notin\Ga_N(\overline{\a})$, every $H$-orbit contained in $\Ga_N(\overline{\a})$ has length divisible by 7. Hence $7\mid p$, and so $p=7$. Thus $\Ga_N(\overline{\a})$ is a single $H$-orbit of length 7.
Let $\overline{\b}=\overline{\a}^g\in\Ga_N(\overline{\a})$, where $g=(g_1,\ldots,g_\ell)\in\overline{G}$. Then $7=|\overline{\b}^H|=\prod_{j=1}^{\ell}|K_j:K_j\cap K_j^{g_j}|$. Each factor belongs to $\{1,7\}$, so, after relabeling if necessary, we may assume that $|K_1:K_1\cap K_1^{g_1}|=7$ and $|K_j:K_j\cap K_j^{g_j}|=1$ for every $j\geqslant2$. Thus $g_j\in\N_{T_j}(K_j)=K_j$ for every $j\geqslant2$. 
Let $M=T_1\times K_2\times\cdots\times K_\ell$. Then $H\leqslant M$ and $g\in M$. The equality $\Ga_N(\overline{\a})=\overline{\b}^H=\overline{\a}^{gH}$ now gives $\Ga_N(\overline{\a})\subseteq\overline{\a}^M$.
For $\overline{u}\in\overline{\a}^M$, let $m\in M$ such that $\overline{u}=\overline{\a}^m$. The $\overline{G}$-invariance of $\Ga_N$ gives $\Ga_N(\overline{u})=\Ga_N(\overline{\a})^m\subseteq (\overline{\a}^M)^m=\overline{\a}^M$. 
Thus no vertex of $\overline{\a}^M$ is adjacent to a vertex outside $\overline{\a}^M$. The connectedness of $\Ga_N$ forces $\overline{\a}^M=V(\Ga_N)$. 
However, $|\overline{\a}^M|=|M:H|=|T_1:K_1|=8$, while $|V(\Ga_N)|=8^\ell>8$ because $\ell\geqslant2$, a contradiction. Hence $\overline{X}$ is not of affine type, completing the proof.\qed

\begin{lemma}\label{lem:AS-large}
Suppose that $\overline{X}$ is of almost simple type. Then $\Ga_N\cong\K_{p+1}$, where $p\geqslant11$.
\end{lemma}

\pf\, Since $\overline{X}$ is of almost simple type, by Lemma \ref{subgroup}, we have $\overline{G}\leqslant L$. Since $\ell\geqslant2$, the group $\overline{G}\cong T^\ell$ is not simple, and so $\overline{G}\neq L$. Therefore, \cite[Theorem 1.3 (1)]{DS2020} applies, leading to that $L\cong\A_m$ and hence $\overline{X}\in\{\A_m,\Sy_m\}$ for $m=|V(\Ga_N)|>10$. 
It follows that $\overline{X}$ is $2$-transitive on $V(\Ga_N)$, and so $\Ga_N\cong\K_m$ and $p=m-1$. Since $p$ is an odd prime and $m>10$, we have $p\geqslant11$ and $m=p+1\geqslant12$. This proves the lemma.\qed

\begin{lemma}\label{lem:simple-diagonal-large}
$\overline{X}$ is not of simple diagonal type.
\end{lemma}
\pf\, Suppose that $\overline{X}$ is of simple diagonal type. Then $(\overline{X})_{\overline{\a}}$ contains $L_{\overline{\a}}\cong Q$, and is therefore nonsolvable. Applying \cite[Propositions 2--5]{DM1980}, \cite[Theorem 1.1]{GF2012}, \cite[Theorem 1.1]{GLH2016}, \cite[Theorem 2.1]{GHX2015}, and \cite[Theorem 1.2]{LLL2018}, the graph $\Ga_N$ is $(\overline{X},2)$-arc-transitive. On the other hand, by \cite[Theorem 2]{P1993}, a $2$-arc-transitive quasiprimitive group cannot be of simple diagonal type, a contradiction. Hence the lemma follows.\qed

The following elementary lemma will be used in the proof of the case where $\overline{X}$ is of product action type.

\begin{lemma}\label{lem:pd}
Let $\Ga_N$ be a connected $p$-valent $L$-arc-transitive graph, where $p$ is an odd prime, and let $\overline{\a}\in V(\Ga_N)$. Suppose that $v_p(|L_{\overline{\a}}|)=1$. Let $r\geqslant2$ and $d$ be positive integers satisfying $d^r\mid |L_{\overline{\a}}|$. Then $L_{\overline{\a}}$ has a subgroup of index $p$ and $p\nmid d$.
\end{lemma}

\begin{proof}
Let $\overline{\b}\in\Ga_N(\overline{\a})$. The $L$-arc-transitivity of $\Ga_N$ implies that $L_{\overline{\a}}$ is transitive on $\Ga_N(\overline{\a})$. The $p$-valency of $\Ga_N$ gives $|\Ga_N(\overline{\a})|=p$. The stabilizer of $\overline{\b}$ in $L_{\overline{\a}}$ is $L_{\overline{\a}\overline{\b}}$, and hence $|L_{\overline{\a}}:L_{\overline{\a}\overline{\b}}|=|\Ga_N(\overline{\a})|=p$. Thus $L_{\overline{\a}}$ contains a subgroup of index $p$, and $p\mid |L_{\overline{\a}}|$. Assume that $p\mid d$. Then $v_p(d)\geqslant1$. The divisibility $d^r\mid |L_{\overline{\a}}|$ gives $v_p(|L_{\overline{\a}}|)\geqslant v_p(d^r)=rv_p(d)$. Together with $r\geqslant2$, this gives $v_p(|L_{\overline{\a}}|)\geqslant2$, contrary to $v_p(|L_{\overline{\a}}|)=1$. Hence $p\nmid d$.
\end{proof}

\begin{lemma}\label{lem:product-action}
Suppose that $\overline{X}$ is of product action type. Then either $\overline{G}=\soc(\overline{X})\unlhd \overline{X}$, or $r=\ell$, $T<Q$, and $(Q,T,r)$ is listed in Table \ref{tab:PA-final}.
\end{lemma}

\pf\, Since $1\neq L=\soc(\overline{X})\unlhd\overline{X}$ and $\overline{X}$ is quasiprimitive on $V(\Ga_N)$, the group $L$ is transitive on $V(\Ga_N)$. Moreover, since $\overline{X}$ is of product action type, \cite[Section 2, Remark 2.1]{P1993} implies that $ L_{\overline{\a}}$ is a nontrivial normal subgroup of $(\overline{X})_{\overline{\a}}$. As $\Ga_N$ is $p$-valent and $\overline{X}$-arc-transitive, $(\overline{X})_{\overline{\a}}^{\Ga_N(\overline{\a})}$ is primitive of degree $p$, and hence $L_{\overline{\a}}^{\Ga_N(\overline{\a})}$ is either trivial or transitive.

Suppose that $L_{\overline{\a}}^{\Ga_N(\overline{\a})}=1$. Then $L_{\overline{\a}}$ fixes every neighbor of $\overline{\a}$. For $\overline{\b}\in\Ga_N(\overline{\a})$, we have $L_{\overline{\a}}\leqslant L_{\overline{\b}}$. Since $L$ is vertex-transitive, $|L_{\overline{\a}}|=|L_{\overline{\b}}|$, and so $L_{\overline{\a}}=L_{\overline{\b}}$. By the connectivity of $\Ga_N$, the subgroup $L_{\overline{\a}}$ fixes every vertex of $\Ga_N$. Since $\overline{X}$, and so $L$, acts faithfully on $V(\Ga_N)$, we obtain $L_{\overline{\a}}=1$, a contradiction. Then $L_{\overline{\a}}^{\Ga_N(\overline{\a})}$ is transitive on $\Ga_N(\overline{\a})$. Thus $\Ga_N$ is $L$-arc-transitive and $L$-locally primitive.

Suppose first that $\overline{X}$ is of product action type arising from simple diagonal action. By \cite[Section 3, pp. 303--304]{BP2003}, the group $L_{\overline{\a}}$ is a direct product of at least two isomorphic nontrivial groups. Since $L_{\overline{\a}}^{\Ga_N(\overline{\a})}$ is transitive, we have $p\mid|L_{\overline{\a}}|$. It follows that $v_p(|L_{\overline{\a}}|)\geqslant2$, and so $v_p(|(\overline{X})_{\overline{\a}}|)\geqslant2$, contrary to the result of Lemma \ref{power}. Thus $\overline{X}$ is of product action type arising from almost simple action.

By \cite[Section 2]{P1993}, the group $L_{\overline{\a}}$ is a subdirect subgroup of $R_1\times\cdots\times R_r$, where $1\neq R_i<Q_i$ and $R_i\cong R_1$ for every $i$. Let $\pi_i:R_1\times\cdots\times R_r\to R_i$ be the natural projection and $\eta_i=\pi_i|_{L_{\overline{\a}}}$. Then each $\eta_i$ is surjective. Let $L_i=\prod_{j\neq i}Q_j$. Then $\ker(\eta_i)=L_i\cap L_{\overline{\a}}=(L_i)_{\overline{\a}}$.

We claim that every $\eta_i$ is injective. On the contrary, we suppose that $\ker(\eta_i)=(L_i)_{\overline{\a}}\neq1$ for some $i$. Then $L_i$ is not semiregular on $V(\Ga_N)$. Since $\Ga_N$ is connected and $L$-locally primitive, Lemma \ref{NQ} (1) implies that $L_i$ has at most two orbits. If $L_i$ had exactly two orbits, then, since $L_i\unlhd L$, these two orbits would form an $L$-invariant partition. As $L$ is transitive, the induced action of $L$ on the two parts would be transitive and so $L$ contains a normal subgroup with index 2, contrary to the fact that $L\cong Q^r$. Hence $L_i$ is transitive. Thus $L=L_iL_{\overline{\a}}$, and so $Q_i\cong L/L_i=L_iL_{\overline{\a}}/L_i\cong L_{\overline{\a}}/(L_i)_{\overline{\a}}=L_{\overline{\a}}/\ker(\eta_i)=L_{\overline{\a}}/\ker(\pi_i|_{L_{\overline{\a}}})\cong R_i$ as $\eta_i=\pi_i|_{L_{\overline{\a}}}$ is surjective, contradicting that assumption that $R_i<Q_i$. Thus every $\eta_i$ is injective, as claimed above. 

Based on the previous discussion, we know that each $\eta_i$ is bijection. Hence $L_{\overline{\a}}$ is a full diagonal subgroup of $R_1\times\cdots\times R_r$, i.e., $L_{\overline{\a}}\cong R_i$ for every $i$ (see Definition \ref{defi} for example). Fix $i$, noting that $R_i\cong L_{\overline{\a}}$, one yields that $1<|R_i|<|Q|$ and $|V(\Ga_N)|=|L:L_{\overline{\a}}|=|Q|^r/|R_i|$. Derived from $\overline{G}$ is transitive on $V(\Ga_N)$, we yield that $|Q|^r/|R_i|$ divides $|T|^\ell$. Pick up arbitrary element $z$ from $\Pi(|Q|)$. It deduce from $R_i$ is isomorphic to a subgroup of $Q$ and $r\geqslant2$ that $v_z(|Q|^r/|R_i|)\geqslant(r-1)v_z(|Q|)>0$. Hence $z\in\Pi(|T|)$ and  $\Pi(|Q|)\subseteq\Pi(|T|)$. 
Lemma \ref{subgroup} gives $\overline{G}\leqslant L\cong Q^r$, which implies $\Pi(|T|)\subseteq\Pi(|Q|)$. Thus  $\Pi(|Q|)=\Pi(|T|)$.

Suppose that $\overline{G}$ contains a minimal normal subgroup $Q_i$ of $L$ for some $i\in\{1,2,\cdots,r\}$. Since $Q_i\unlhd L$ and $Q_i\leqslant\overline{G}\leqslant L$, we have $Q_i\unlhd\overline{G}$. As $Q_i$ is simple and $\overline{G}\cong T^\ell$, it follows that $Q_i\cong T$, and so $Q\cong T$. Since $\overline{G}$ is transitive, $L=\overline{G}L_{\overline{\a}}$ and so $|T|^r=|T|^\ell|L_{\overline{\a}}|/|\overline{G}\cap L_{\overline{\a}}|$. As $|L_{\overline{\a}}|=|R_i|<|T|$, we obtain $r\leqslant\ell$. On the other hand, $|\overline{G}|=|T|^\ell$ divides $|L|=|T|^r$, and so $\ell\leqslant r$. Thus $r=\ell$ and $|\overline{G}|=|L|$, so $\overline{G}=L=\soc(\overline{X})\unlhd\overline{X}$. Thus the first conclusion of the lemma follows.

Assume now that $\overline{G}$ contains no minimal normal subgroup of $L$. By \cite[Theorem 1.3 (3)]{DS2020}, one of the following holds:
\begin{enumerate}[leftmargin=1.2em,label=(\alph*)]
\item $\ell=r$, $\overline{G}=T_1\times\cdots\times T_r$ with $T_i<Q_i$ for every $i\in\{1,2,\cdots,r\}$.
\item The $T_i$ are pairwise disjoint strips in $L$, $|\mathcal{P}(T_i)|\in\{1,2\}$ for every $i$, $|\mathcal{P}(T_i)|=2$ for at least one $i$, and $(T,Q)$ is listed in Table \ref{tab:groups}.
\item $L\cong(\A_t)^r$, $|V(\Ga_N)|=t^r$, and $t\geqslant10$.
\end{enumerate}

\begin{table}[htbp]
\centering
\renewcommand{\arraystretch}{1.3}
\setlength{\tabcolsep}{10pt}
\caption{The pairs $(T,Q)$ arising in case (b).}
\label{tab:groups}
{\scriptsize
\begin{tabular}{|c|c|c|}
\hline
$T$ & $Q$ & remark \\
\hline
$\A_6$    & $\A_6$ & \\
\hline
$\M_{12}$ & $\M_{12}$, $\A_{12}$ & \\
\hline
$\POmega_8^+(q)$ & $\POmega_8^+(q)$ & $q\geqslant2$ \\
& $\A_n$           & $n=|\POmega_8^+(q):\Omega_7(q)|$ \\
& $\Sp_8(2)$       & $q=2$ \\
\hline
$\Sp_4(2^a)$     & $\Sp_{4b}(2^{a/b})$ & $b\mid a$, $a\geqslant2$ \\
& $\A_n$              & $n=|\Sp_4(2^a):\Sp_2(2^{2a}).\ZZ_2|$, $a\geqslant2$ \\
\hline
\end{tabular}}
\end{table}

We first exclude (b). Since the supports $\mathcal{P}(T_i)$ are pairwise disjoint and $|\mathcal{P}(T_i)|=2$ for at least one $i$, we have $r\geqslant\sum_{i=1}^{\ell}|\mathcal{P}(T_i)|\geqslant\ell+1$, and so $\ell\leqslant r-1$. If $T\cong Q$, then $|\overline{G}|=|Q|^\ell\leqslant|Q|^{r-1}<|Q|^r/|R_i|=|V(\Ga_N)|$, contrary to the transitivity of $\overline{G}$. Thus $T\not\cong Q$. It remains to consider the non-isomorphic pairs in Table \ref{tab:groups}.

For $(Q,T)=(\A_{12},\M_{12})$, we have $7\in\Pi(|Q|)\setminus\Pi(|T|)$. For $(Q,T)=(\Sp_8(2),\POmega_8^+(2))$, we get that $17\in\Pi(|Q|)\setminus\Pi(|T|)$. Let $(Q,T)=(\A_n,\POmega_8^+(q))$, where $n=|\POmega_8^+(q):\Omega_7(q)|=\frac{(2,q-1)}{(4,q^4-1)}q^3(q^4-1)$. Every prime divisor of $|T|$ is at most $q^6-1$. If $q=2$, then $n=120$ and $11\in\Pi(|\A_{120}|)\setminus\Pi(|\POmega_8^+(2)|)$. If $q=3$, then $n=1080$ and $11\in\Pi(|\A_{1080}|)\setminus\Pi(|\POmega_8^+(3)|)$. Assume that $q\geqslant4$. Then $n=q^3(q^4-1)>2(q^6-1)$ whenever $q$ is even and $n=q^3(q^4-1)/2>2(q^6-1)$ whenever $q$ is odd. By \cite[Theorem 418]{HW1960}, there is a prime $z$ such that $q^6-1<z<2(q^6-1)<n$. Hence $z\in\Pi(|\A_n|)\setminus\Pi(|T|)$.

Suppose that $(Q,T)=(\Sp_{4b}(2^{a/b}),\Sp_4(2^a))$, where $b\mid a$ and $a,b\geqslant2$, and put $q_0=2^{a/b}$. Then $|Q|=q_0^{4b^2}\prod_{i=1}^{2b}(q_0^{2i}-1)$ and $|T|=q_0^{4b}(q_0^{2b}-1)(q_0^{4b}-1)$. If $(b,q_0)=(2,2)$, then $7\in\Pi(|Q|)\setminus\Pi(|T|)$. Suppose that $b=2$ and $q_0>2$. By Lemma \ref{lem:ppd}, there is a primitive prime divisor $z$ of $q_0^6-1$. Then $z\mid|Q|$, but $z\nmid q_0^4-1$ and $z\nmid q_0^8-1$, so $z\notin\Pi(|T|)$. Suppose now that $b\geqslant3$. By \cite[Theorem 418]{HW1960}, there is a prime $w$ with $b<w<2b$. By Lemma \ref{lem:ppd}, there is a primitive prime divisor $z$ of $q_0^{2w}-1$ and $z\mid|Q|$. It follows from $2w\nmid2b$ and $2w\nmid4b$ that $z\nmid|T|$, and so $\Pi(|Q|)\neq\Pi(|T|)$. Assume that $(Q,T)=(\A_n,\Sp_4(2^a))$ for $q=2^a\geqslant4$ and $n=|\Sp_4(q):\Sp_2(q^2).\ZZ_2|$. Since $|\Sp_4(q)|=q^4(q^2-1)(q^4-1)$ and $|\Sp_2(q^2).\ZZ_2|=2q^2(q^4-1)$, we have $n=q^2(q^2-1)/2$. Every prime divisor of $|T|$ is at most $q^2+1$, and $n>2(q^2+1)$. By \cite[Theorem 418]{HW1960}, there is a prime $z$ such that $q^2+1<z<2(q^2+1)<n$ and so $z\in\Pi(|\A_n|)\setminus\Pi(|T|)$.

The preceding two paragraphs show that  $\Pi(|Q|)\neq\Pi(|T|)$ in every remaining possibility, a contradiction arisen as $\Pi(|Q|)=\Pi(|T|)$. Hence the case (b) does not occur.

We assert that the case (c) is impossible. On the contrary, we assume that the case (c) holds, that is $L\cong(\A_t)^r$, $|V(\Ga_N)|=t^r$ and $t\geqslant10$. Derived from $\Gamma_N$ is $p$-valent $L$-arc-transitive, we deduce that $p\mid|L_{\overline{\a}}|$. And because $|L_{\overline{\a}}|=|L|/|V(\Ga_N)|=|(\A_t)^r|/t^r=((t-1)!/2)^r$ and $r\geqslant 2$, one yields that $v_p(|L_{\overline{\a}}|)\geqslant r\geqslant 2$, which contrary to the fact that $v_p(|L_{\overline{\a}}|)=1$. Hence the assertion holds.

Now that case (a) holds, that is, $\ell=r$ and $\overline{G}=T_1\times\cdots\times T_r$, where $T_i<Q_i$ for every $i\in\{1,2,\ldots,r\}$. Let $d=|Q:T|$. The relations $L\cong Q^r$ and $\overline{G}\cong T^r$ give $|L|/|\overline{G}|=|Q|^r/|T|^r=d^r$. 
Moreover, $\overline{G}\leqslant L$ is transitive on $V(\Ga_N)$, yielding $L=\overline{G}L_{\overline{\a}}$ and so $|L|/|\overline{G}|=|\overline{G}L_{\overline{\a}}|/|\overline{G}|= |\overline{G}||L_{\overline{\a}}|/(|\overline{G}\cap L_{\overline{\a}}||\overline{G}|) =|L_{\overline{\a}}|/|L_{\overline{\a}}\cap \overline{G}|$. 
Thus $d^r=|L_{\overline{\a}}|/|L_{\overline{\a}}\cap \overline{G}|$, which, together with $L_{\overline{\a}}\cong R_i$, gives $d^r\mid |R_i|$. 
Recalling that $R_i<Q_i$ and $r\geqslant2$, the divisibilities $d^r\mid |R_i|$ and $|R_i|\mid |Q_i|$ imply $d^r\mid |Q_i|$. It follows that $2\leqslant r\leqslant rv_z(d)\leqslant v_z(|Q_i|)$ for every prime $z\mid d$. Together with $\Pi(|T|)=\Pi(|Q|)$, all the hypotheses of Lemma \ref{lem:equal-prime-divisors} are satisfied. Hence $(Q,T,r)$ occurs in Tables \ref{tab:2.1} and \ref{tab:2.2}.

We first consider the candidates listed in Table \ref{tab:2.1}. For the last row of Table \ref{tab:2.1}, $(Q_i,T_i)=(\PSp_4(q),\PSL_2(q^2))$ and $q>2$. The group orders are $|Q_i|=q^4(q^2-1)(q^4-1)/(2,q-1)$ and $|T_i|=q^2(q^4-1)/(2,q^2-1)$. 
The equality $(2,q^2-1)=(2,q-1)$ gives $d=|Q_i:T_i|=q^2(q^2-1)$. Hence $|Q_i|/d^2=(q^4(q^2-1)(q^4-1)/(2,q-1))/(q^4(q^2-1)^2)=(q^4-1)/((2,q-1)(q^2-1))=(q^2+1)/(2,q-1)$. Suppose first that $q>3$. By \cite[Theorem 5.2.2 and Table 5.2. A]{KL1990}, every proper subgroup of $\PSp_4(q)$ has index at least $(q^4-1)/(q-1)=(q+1)(q^2+1)$. In particular, $R_i<Q_i\cong\PSp_4(q)$ gives $|Q_i:R_i|\geqslant(q+1)(q^2+1)$. 
Together with $d^2\mid |R_i|$, this gives $d^2(q+1)(q^2+1)\leqslant |R_i||Q_i:R_i|=|Q_i|$,  $(q+1)(q^2+1)\leqslant |Q_i|/d^2=(q^2+1)/(2,q-1)$, i.e., $(q+1)\leqslant 1/(2,q-1)$, which is impossible. For $q=3$, every proper subgroup of $\PSp_4(3)$ has index at least 27 by \cite{Atlas}. Hence $|Q_i:R_i|\geqslant27$, and $d^2\mid |R_i|$ gives $|Q_i|/d^2\geqslant|Q_i:R_i|\geqslant 27$. 
However, $|Q_i|/d^2=(q^2+1)/(2,q-1)=(3^2+1)/(2,3-1)=5$, which is impossible. Thus the last row of Table \ref{tab:2.1} does not occur. 
For rows 2 and 3 of Table \ref{tab:2.1}, the case $m=2$ gives $\POmega_5(q)\cong\PSp_4(q)$ and $\POmega^-_4(q)\cong\PSL_2(q^2)$, respectively, which reduces to the last row considered above. 
Thus $m\geqslant4$ in rows 2 and 3. The possibilities arising from Table \ref{tab:2.1} are precisely those listed in rows 1--4 of Table \ref{tab:PA-final}.

Now we consider the candidates listed in Table \ref{tab:2.2}. Suppose first that $(Q_i,T_i,r)=(\PSU_4(2),\A_6,2)$. Then $d=|Q_i|/|T_i|=25920/360=72$, giving $|Q_i|/d^2=25920/72^2=5$. Let $M$ be a maximal subgroup of $Q_i$ containing $R_i$. The inclusion $R_i\leqslant M<Q_i$ gives $|Q_i:M|\mid |Q_i:R_i|$, while $d^2\mid |R_i|$ gives $|Q_i:R_i|\mid |Q_i|/d^2=5$. Hence $|Q_i:M|=|Q_i:R_i|=5$, but \cite{Atlas} shows that $\PSU_4(2)$ has no maximal subgroup of index $5$. The same argument applies to $(Q_i,T_i,r)=(\PSL_6(2),\PSL_5(2),2)$ and $(\POmega_8^+(2),\A_9,2)$. In these two cases, $|Q_i|/d^2=4960$ and $189$, respectively, whereas \cite{Atlas} shows that no maximal subgroup of the corresponding group $Q_i$ has index dividing these numbers. Hence neither case occurs.

Assume that $(Q_i,T_i,r)=(\Sp_6(2),\A_8,2)$, as listed in Table \ref{tab:2.2}. The group orders $|Q_i|=1451520$ and $|T_i|=20160$ give $d=|Q_i:T_i|=72$ and $|Q_i|/d^2=280$. Choose a maximal subgroup $M$ of $Q_i$ containing $R_i$. The relations $R_i\leqslant M<Q_i$ and $d^2\mid |R_i|$ give $|Q_i:M|\mid |Q_i:R_i|$ and $|Q_i:R_i|\mid280$. Thus $|Q_i:M|\mid280$. By \cite{Atlas}, the only maximal subgroup of $\Sp_6(2)$ whose index divides 280 is $M\cong\PSU_4(2).\ZZ_2$, with $|Q_i:M|=28$ and $|M|=51840=2\cdot5\cdot72^2$. The isomorphism $L_{\overline{\a}}\cong R_i$ and the divisibility $d^2\mid |R_i|$ allow Lemma \ref{lem:pd} to be applied with $r=2$. Thus $R_i$ contains a subgroup of index $p$, where $p\mid |R_i|$ and $p\nmid d$. The inclusion $R_i\leqslant M$ gives $p\mid |M|$. The prime divisors of $|M|$ are 2, 3 and 5, while $p\nmid72$, so $p=5$. The relations $72^2\mid |R_i|$ and $5\mid |R_i|$, together with $(5,72)=1$, give $5\cdot72^2=25920\mid |R_i|$. The graph $\Ga_N$ is connected pentavalent $L$-arc-transitive. By \cite[Theorem 1.1]{GF2012}, its vertex stabilizer satisfies $|L_{\overline{\a}}|\leqslant23040$. This contradicts $L_{\overline{\a}}\cong R_i$ and $25920\mid |R_i|$. Thus $(Q_i,T_i,r)=(\Sp_6(2),\A_8,2)$ does not occur.

Assume next that $(Q_i,T_i,r)=(\M_{11},\PSL_2(11),2)$. Here $d=|Q_i:T_i|=12$ and $|Q_i|/d^2=7920/12^2=55$. Choose a maximal subgroup $M$ of $Q_i$ containing $R_i$. The relations $R_i\leqslant M<Q_i$ and $12^2\mid |R_i|$ give $|Q_i:M|\mid |Q_i:R_i|$ and $|Q_i:R_i|\mid55$. Thus $|Q_i:M|\mid55$. By \cite{Atlas}, the maximal-subgroup indices of $\M_{11}$ dividing 55 are 11 and 55.
Suppose first that $|Q_i:M|=55$. Then $|M|=144$. The divisibility $12^2\mid |R_i|$ and the inclusion $R_i\leqslant M$ force $|R_i|=144$. Lemma \ref{lem:pd} gives $p\mid |R_i|$ and $p\nmid12$. This is impossible because $|R_i|=144=2^4\cdot3^2$. Thus $|Q_i:M|=11$. By \cite{Atlas}, $M\cong\M_{10}$ and $|M|=720$. The relations $p\mid |R_i|$, $R_i\leqslant M$ and $p\nmid12$ force $p=5$. Together with $12^2\mid |R_i|$ and $(5,12)=1$, this gives $5\cdot12^2=720\mid |R_i|$. The equality $|M|=720$ forces $R_i=M\cong\M_{10}$.
The isomorphism $L_{\overline{\a}}\cong R_i$ gives $|L_{\overline{\a}}|=720$ and $L_{\overline{\a}}\cong\M_{10}$. By \cite[Theorem 1.1]{GF2012}, a vertex stabilizer of order $720$ in a connected pentavalent arc-transitive graph is isomorphic to $\A_4\times\A_5$. This contradicts $\M_{10}\not\cong\A_4\times\A_5$. Thus $(Q_i,T_i,r)=(\M_{11},\PSL_2(11),2)$ does not occur.

We now consider the case $(Q_i,T_i)=(\M_{12},\M_{11})$ listed in Table \ref{tab:2.2}. Then $d=|Q_i:T_i|=12$ and $r\in\{2,3\}$. The isomorphism $L_{\overline{\a}}\cong R_i$ and the divisibility $d^r\mid |R_i|$ allow Lemma \ref{lem:pd} to be applied. Thus $R_i$ contains a subgroup of index $p$, $p\mid |R_i|$ and $p\nmid d=12$. 
Suppose first that $r=3$. Let $M$ be a maximal subgroup of $Q_i$ containing $R_i$. The divisibility $12^3\mid |R_i|$ gives $|Q_i:R_i|\mid |Q_i|/12^3=55$, and hence $|Q_i:M|\mid55$. Thus $|Q_i:M|\in\{5,11,55\}$. However, by \cite{Atlas}, $\M_{12}$ has no maximal subgroup of any of these indices. Hence $r=2$. In this case, $|Q_i:M|\mid |Q_i|/12^2=660$, and \cite{Atlas} gives $M\cong\M_{11}$, $\A_6.\ZZ_2^2$, or $\ZZ_3^2.\ZZ_2.\Sy_4$. The relation $p\nmid12$ excludes $p=2$ and $p=3$, while $p\mid |R_i|$ and $R_i\leqslant M$ imply $p\mid |M|$. If $M\cong\ZZ_3^2.\ZZ_2.\Sy_4$, then $\Pi(|M|)=\{2,3\}$, a contradiction. Thus $M\cong\ZZ_3^2.\ZZ_2.\Sy_4$ is impossible. 

Suppose that $M\cong\A_6.\ZZ_2^2$. Then $\Pi(|M|)=\{2,3,5\}$, and the conditions $p\mid |M|$ and $p\nmid12$ give $p=5$. Moreover, $12^2\mid |R_i|$ and $5\mid |R_i|$. As $(5,12)=1$, it follows that $5\cdot12^2=720\mid |R_i|$. The inclusion $R_i\leqslant M$ and $|M|=1440$ then give $|R_i|\in\{720,1440\}$. Moreover, $|M:R_i|\leqslant2$ implies $\soc(M)\cong\A_6\leqslant R_i$. However, by \cite[Theorem 1.1]{GF2012}, a pentavalent vertex stabilizer of order 720 or 1440 is isomorphic to $\A_4\times\A_5$ or $(\A_4\times\A_5)\rtimes\ZZ_2$, respectively, neither of which has a composition factor isomorphic to $\A_6$. Hence this case does not occur.

It remains to consider $M\cong\M_{11}$. Then $R_i\leqslant M\cong\M_{11}$. The relations $p\mid |R_i|$ and $p\nmid12$, together with $|\M_{11}|=2^4\cdot3^2\cdot5\cdot11$, give $p\in\{5,11\}$. Suppose first that $p=5$. Then $5\mid |R_i|$, while $d^2=12^2\mid |R_i|$. As $(5,12)=1$, we obtain $5\cdot12^2=720\mid |R_i|$. 
By \cite[Theorem 1.1]{GF2012}, the pentavalent vertex stabilizer $L_{\overline{\a}}\cong R_i$ is one of the groups listed there. Among these candidates, the conditions $R_i\leqslant\M_{11}$ and $720\mid |R_i|$ leave only $R_i\cong\A_4\times\A_5$, which has order 720. Hence $|\M_{11}:R_i|=7920/720=11$. 
However, \cite{Atlas} shows that a subgroup of index 11 in $\M_{11}$ is isomorphic to $\M_{10}\cong\A_6.\ZZ_2$, a contradiction. Thus $p\neq 5$. It follows that $p=11$. In this case, $11\mid |R_i|$ and $12^2\mid |R_i|$. 
Since $(11,12)=1$, we obtain $11\cdot12^2=1584\mid |R_i|$. The inclusion $R_i\leqslant\M_{11}$ gives $|R_i|\mid7920$, and $7920=5\cdot1584$. Thus $|\M_{11}:R_i|=7920/|R_i|$ divides 5. By \cite{Atlas}, $\M_{11}$ has no proper subgroup of index 5, so $R_i=\M_{11}$. Thus $p=11$ and $L_{\overline{\a}}\cong R_i\cong\M_{11}$.

The preceding discussion shows that $L=Q_1\times Q_2$, $\overline{G}=T_1\times T_2$ and $L_{\overline{\a}}\cong\M_{11}$, where $Q_i\cong\M_{12}$ and $T_i\cong\M_{11}$ for $i=1,2$. 
Set $H_1=L_{\overline{\a}}\cap(T_1\times Q_2)$ and $H_2=L_{\overline{\a}}\cap(Q_1\times T_2)$. The inclusion $\overline{G}\leqslant T_1\times Q_2$, together with $L=\overline{G}L_{\overline{\a}}$, implies $L=(T_1\times Q_2)L_{\overline{\a}}$. 
The product formula gives $|L|=|T_1\times Q_2||L_{\overline{\a}}|/|(T_1\times Q_2)\cap L_{\overline{\a}}|$. Dividing by $|T_1\times Q_2|$ yields $|L:T_1\times Q_2|=|L_{\overline{\a}}:(T_1\times Q_2)\cap L_{\overline{\a}}|=|L_{\overline{\a}}:H_1|$. Also, $|L:T_1\times Q_2|=|Q_1\times Q_2:T_1\times Q_2|=|Q_1:T_1|=12$. Hence $|L_{\overline{\a}}:H_1|=12$. Reversing the roles of the two direct factors gives $|L_{\overline{\a}}:H_2|=12$. Note that $H_1\cap H_2=L_{\overline{\a}}\cap(T_1\times Q_2)\cap(Q_1\times T_2)=L_{\overline{\a}}\cap(T_1\times T_2)=L_{\overline{\a}}\cap\overline{G}$. 
The factorization $L=\overline{G}L_{\overline{\a}}$ gives $|L|=|\overline{G}||L_{\overline{\a}}|/|\overline{G}\cap L_{\overline{\a}}|$. Dividing by $|\overline{G}|$ yields $|L:\overline{G}|=|L_{\overline{\a}}:L_{\overline{\a}}\cap\overline{G}|=|L_{\overline{\a}}:H_1\cap H_2|$. On the other hand, $|L:\overline{G}|=|Q_1\times Q_2:T_1\times T_2|=|Q_1:T_1||Q_2:T_2|=12^2$. Hence $|L_{\overline{\a}}:H_1\cap H_2|=12^2$.
Together with $|L_{\overline{\a}}|=|\M_{11}|=7920$, this forces $|H_1\cap H_2|=55$. Both $H_1$ and $H_2$ have index 12 in $L_{\overline{\a}}\cong\M_{11}$. By \cite{Atlas}, all subgroups of index 12 in $\M_{11}$ are conjugate, and the action of $\M_{11}$ on the 12 right cosets of such a subgroup is 2-transitive. The equality $H_1=H_2$ would give $|H_1\cap H_2|=|H_1|=660$, which is incompatible with $|H_1\cap H_2|=55$. Thus $H_1\neq H_2$. Choose $g\in L_{\overline{\a}}$ such that $H_2=H_1^g$. In the action of $L_{\overline{\a}}$ on the right cosets of $H_1$, the stabilizers of $H_1$ and $H_1g$ are $H_1$ and $H_2$, respectively. These two points are distinct, and $2$-transitivity implies that $H_1$ is transitive on the remaining $11$ points. The stabilizer of $H_1g$ in $H_1$ is $H_1\cap H_2$, and hence $|H_1:H_1\cap H_2|=11$. It follows that $|H_1\cap H_2|=660/11=60$, contradicting $|H_1\cap H_2|=55$. Thus $(Q_i,T_i)=(\M_{12},\M_{11})$ does not occur.

Finally, we turn to the case $(Q_i,T_i)=(\M_{24},\M_{23})$ and $r\in\{2,3\}$ listed in Table \ref{tab:2.2}. In this case, $d=|Q_i:T_i|=24$ and $d^r\mid |R_i|$. The relations $p\mid |R_i|$, $R_i\leqslant Q_i\cong\M_{24}$ and $p\nmid24$ give $p\in\{5,7,11,23\}$. Suppose that $r=3$. Then $24^3\mid |R_i|$. By the classifications of vertex stabilizers of connected arc-transitive graphs of valency 5, 7, and prime valency greater than 7 in \cite[Theorem 1.1]{GF2012}, \cite[Theorem 1.1]{GLH2016} and \cite[Theorems 1.1 and 1.2]{LLL2018}, respectively, none of the possible vertex stabilizers $R_i$ with $p\in\{5,7,11,23\}$ has order both divisible by $24^3$ and dividing $|\M_{24}|$. Thus $r=2$ and so $24^2\mid |R_i|$. Suppose first that $p=5$. By \cite[Theorem 1.1]{GF2012}, together with $24^2\mid |R_i|$ and $|R_i|\mid|\M_{24}|$, the possible vertex stabilizers are $R_i\in\{\Sy_4\times\Sy_5, \AGL_2(4),\AGammaL_2(4), \ZZ_2^6\rtimes\GammaL_2(4)\}$. Suppose next that $p=7$. By \cite[Theorem 1.1]{GLH2016}, the conditions $24^2\mid |R_i|$ and $|R_i|\mid|\M_{24}|$ leave only $R_i\in\{\PSL_3(2)\times\Sy_4, \ZZ_2^6\rtimes(\SL_2(2)\times\SL_3(2))\}$. For $p=11$, the vertex-stabilizer classification in \cite[Theorems 1.1 and 1.2]{LLL2018} contains no group whose order is divisible by $24^2$ and divides $|\M_{24}|$, excluding this possibility. Finally, suppose that $p=23$. The same classification, together with $24^2\mid |R_i|$ and $|R_i|\mid|\M_{24}|$, gives $R_i\cong\M_{23}$. On the other hand, $L=\overline{G}L_{\overline{\a}}$ gives $Q_i=T_iR_i$, and hence $|R_i:R_i\cap T_i|=|Q_i:T_i|=24$. This would give a subgroup of index $24$ in $R_i\cong\M_{23}$, which does not exist by \cite{Atlas} and so $p\neq23$. This yields row 5 of Table \ref{tab:PA-final}.

Combining the above arguments, we obtain that either $\overline{G}\unlhd\overline{X}$, or $(Q,T,r)$ is one of the cases listed in Table \ref{tab:PA-final}. This completes the proof. \qed

\begin{table}[htbp]
\centering
\renewcommand{\arraystretch}{1.3}
\setlength{\tabcolsep}{8pt}
\caption{The possibilities for $(Q,T,r)$.}
\label{tab:PA-final}
{\scriptsize
\begin{tabular}{|c|c|c|p{7cm}|}
\hline
$Q$ & $T$ & $r$ & Remark\\
\hline
$\A_c$ & $\A_k$ & $r\geqslant2$ & $5\leqslant k<c$, $(k,c]$ contains no prime, $d=c!/k!$, $rv_z(d)\leqslant v_z(|\A_c|)$ for every prime $z\mid d$\\
\hline
$\PSp_{2m}(q)$ & $\POmega^-_{2m}(q)$ & $2$ & $m\geqslant4$ is even and $q$ is even\\
\hline
$\POmega_{2m+1}(q)$ & $\POmega^-_{2m}(q)$ & $2$ & $m\geqslant4$ is even and $q$ is odd\\
\hline
$\POmega^+_{2m}(q)$ & $\POmega_{2m-1}(q)$ & $2$ & $m\geqslant4$ is even\\
\hline
$\M_{24}$ & $\M_{23}$ & $2$ & $p=5$ and $L_{\overline{\a}}\in\{\Sy_4\times\Sy_5, \AGL_2(4), \AGammaL_2(4), \ZZ_2^6\rtimes\GammaL_2(4)\}$, or $p=7$ and $L_{\overline{\a}}\in\{\PSL_3(2)\times\Sy_4, \ZZ_2^6\rtimes(\SL_2(2)\times\SL_3(2))\}$.\\
\hline
\end{tabular}}
\end{table}

\begin{lemma}\label{lem:twisted-wreath-large}
Suppose that $\overline{X}$ is of twisted wreath action type. Then $\overline{G}=\soc(\overline{X})\unlhd\overline{X}$, and $\Ga_N$ is an $\overline{X}$-normal Cayley graph of $\overline{G}$.
\end{lemma}

\pf\, The twisted wreath action type of $\overline{X}$ implies that $\soc(\overline{X})$ is regular on $V(\Ga_N)$. By Lemma \ref{subgroup}, $\overline{G}\leqslant \soc(\overline{X})$. The transitivity of $\overline{G}$ on $V(\Ga_N)$, together with the regularity of $\soc(\overline{X})$, gives $\overline{G}=\soc(\overline{X})\unlhd\overline{X}$. Hence $\Ga_N$ is an $\overline{X}$-normal Cayley graph of $\overline{G}$, completing the proof. \qed

The preceding lemmas now allow us to complete the proof of Theorem \ref{thm:non-quasiprimitive}.

\begin{proof}[\textbf{Proof of Theorem \ref{thm:non-quasiprimitive}}]
Let $\Ga=\Cay(G,S)$ be a connected $p$-valent $X$-arc-transitive Cayley graph of $G\cong T^k$, where $k\geqslant2$, each $T_i\cong T$ is simple, $p$ is an odd prime, and $G<X\leqslant\Aut(\Ga)$. Suppose that $X$ is not quasiprimitive on $V(\Ga)$, and let $N$ be a maximal intransitive normal subgroup of $X$. Set $\overline{X}=X/N$ and $\overline{G}=GN/N$. Now let $\soc(\overline{X})\cong Q^r$ for some $r\geqslant1$ and simple group $Q$. 
By Lemma \ref{lem:quotient-reduction}, the group $\overline{X}$ is quasiprimitive on $V(\Ga_N)$ and $\overline{G}\cong T^\ell$ for some $1\leqslant\ell\leqslant k$. 

Assume first that $G$ is abelian. Lemma \ref{lem:quotient-reduction} (ii) gives Theorem \ref{thm:non-quasiprimitive} (1), and whenever $\Ga_N\ncong\K_2$, the graph $\Ga$ is an $N$-cover of $\Ga_N$.

In what follows, assume that $G$ is nonabelian. By Lemma \ref{lem:quotient-reduction} (i), the graph $\Ga$ is an $N$-cover of $\Ga_N$. If $\ell=1$, Lemma \ref{lem:simple-quotient} gives Theorem \ref{thm:non-quasiprimitive} (2) and (3), so it remains to consider $\ell\geqslant2$. The quasiprimitivity of $\overline{X}$ on $V(\Ga_N)$ and \cite[Theorem 1]{P1993} imply that $\overline{X}$ is of affine, almost simple, simple diagonal, product action, or twisted wreath action type.

By Lemma \ref{lem:affine-large}, $\overline{X}$ is not of affine type. For the almost simple type, Lemma \ref{lem:AS-large} shows that $\Ga_N$ is a complete graph, yielding Theorem \ref{thm:non-quasiprimitive} (3). The simple diagonal case is excluded by Lemma \ref{lem:simple-diagonal-large}. Suppose next that $\overline{X}$ is of product action type. By Lemma \ref{lem:product-action}, either $\overline{G}\unlhd\overline{X}$ or $(Q,T,r)$ is listed in Table \ref{tab:PA-final}, which gives Theorem \ref{thm:non-quasiprimitive} (4). Finally, assume that $\overline{X}$ is of twisted wreath action type. Lemma \ref{lem:twisted-wreath-large} implies that $\overline{G}=\soc(\overline{X})$ and that $\Ga_N$ is a Cayley graph of $\overline{G}$, establishing Theorem \ref{thm:non-quasiprimitive} (2).

The above cases exhaust all possibilities for $\overline{X}$, and so Theorem \ref{thm:non-quasiprimitive} follows.
\end{proof}

\section{Proofs of the main theorems}\label{sec5}

Let $\Ga$ be a connected $p$-valent $(X,s)$-transitive Cayley graph of characteristically simple group $G=T_1\times\cdots\times T_k\cong T^k$, where each $T_i\cong T$ is simple, $s$ is a positive integer, $p$ is an odd prime, $k\geqslant 2$ is an integer and $G\leqslant X\leqslant \Aut(\Ga)$. Let $\a\in V(\Ga)$. We are now ready to prove one of the main theorems.

\textbf{Proof of Theorem \ref{thm:main}:}

\begin{proof}
Let $\Ga=\Cay(G,S)$ be a connected $p$-valent $X$-arc-transitive Cayley graph of $G\cong T^k$, where $k\geqslant2$, $T$ is simple, $p$ is an odd prime, and $G<X\leqslant\Aut(\Ga)$. We consider the quasiprimitive and non-quasiprimitive cases separately.

First assume that $X$ is quasiprimitive on $V(\Ga)$. By Theorem \ref{thm:quasiprimitive}, the group $X$ is of affine, almost simple, or twisted wreath action type. In the affine case, $G\cong\soc(X)\cong\ZZ_2^k$. For the almost simple type, $\Ga\cong\K_{2^k}$. Finally, Theorem \ref{thm:quasiprimitive} (3) shows that, in the twisted wreath action case, $\Ga$ is a Cayley graph of $\soc(X)$ with connection set consisting of involutions. These are precisely the quasiprimitive possibilities stated in the theorem.

We next treat the non-quasiprimitive case. Let $N$ be a maximal intransitive normal subgroup of $X$, and put $\overline{X}=X/N$, $\overline{G}=GN/N\cong T^\ell$ and $\soc(\overline{X})\cong Q^r$, where $Q$ is simple and $\ell,\,r\geqslant1$. By Theorem \ref{thm:non-quasiprimitive}, the group $\overline{X}$ is quasiprimitive on $V(\Ga_N)$, and only the affine, almost simple, product action and twisted wreath action types can occur. Moreover, $\Ga$ is an $N$-cover of $\Ga_N$ whenever $\Ga_N\ncong\K_2$.
Theorem \ref{thm:non-quasiprimitive} (1) states that $\overline{G}\cong\ZZ_2^\ell$ and that $\Ga_N$ is a Cayley graph of $\overline{G}$. By Lemma \ref{lem:affine-large}, the affine case occurs only when $\ell=1$, and then Theorem \ref{thm:non-quasiprimitive} (2) gives $\Ga_N\cong\K_8$. The same theorem shows that, in the twisted wreath action case, $\Ga_N$ is a Cayley graph of $\soc(\overline{X})$.
The almost simple type is described by Theorem \ref{thm:non-quasiprimitive} (3), according to which one of the following occurs: $\Ga_N$ is complete, $\overline{G}=T=\soc(\overline{X})\unlhd\overline{X}$, or $(\soc(\overline{X}),T)$ is one of the pairs listed in Tables \ref{tab:simple-nonclassical} and \ref{tab:simple-classical-nonsolv}.
Finally, for the product action type, Theorem \ref{thm:non-quasiprimitive} (4) yields either $\overline{G}=\soc(\overline{X})\unlhd\overline{X}$ or $(Q,T,r)$ listed in Table \ref{tab:PA-final}.

The quasiprimitive and non-quasiprimitive cases exhaust all possibilities for $X$, completing the proof.
\end{proof}

We now consider the cubic case. Let $\Ga=\Cay(G,S)$ be a connected cubic symmetric Cayley graph of $G=T_1\times\cdots\times T_k\cong T^k$, where $k\geqslant2$ and each $T_i\cong T$ is a finite nonabelian simple group. Let $A=\Aut(\Ga)$ and let $\a\in V(\Ga)$. Then $G$ acts regularly on $V(\Ga)$ and $\Ga$ is $A$-arc-transitive. The second main theorem is proved next.

\textbf{Proof of Theorem \ref{cubic}:}

\begin{proof}
Let $K=\Core_A(G)$. If $K=G$, then $G\unlhd A$, and the result follows. Suppose next that $K<G$. The relation $K\unlhd G=T_1\times\cdots\times T_k$, together with the simplicity of $T_i$, allows us to apply Lemma \ref{obs}. After relabeling the direct factors if necessary, we may assume that $K=T_1\times\cdots\times T_r$ for some $0\leqslant r<k$. Hence $G/K\cong T^{k-r}$ and  $|G:K|\geqslant|T|\geqslant60$.

The regular action of $G$ on $V(\Ga)$ implies that $K$ is semiregular and has $|G:K|>2$ orbits. By Lemma \ref{NQ}, $\Ga$ is a $K$-cover of $\Ga_K$, the group $G/K$ acts regularly on $V(\Ga_K)$, and $A/K\leqslant\Aut(\Ga_K)$. Moreover, $\Core_{A/K}(G/K)=\Core_A(G)/K=1$. By \cite[Section 2]{DM1980}, $\Ga$ is $(A,s)$-transitive for some $1\leqslant s\leqslant5$. The inequality $|G:K|>2$, together with \cite[Corollary 2.2]{LL2009}, gives $s\geqslant2$ and shows that $\Ga_K$ is a connected cubic core-free $(A/K,s)$-transitive Cayley graph of $G/K$. By \cite[Theorem 1.1 and Table 1]{LL2009}, $G/K\cong\A_{47}$, $T\cong\A_{47}$ and $k-r=1$, and $\Ga_K$ is one of the two cubic $5$-arc-regular Cayley graphs of $\A_{47}$ with $\Aut(\Ga_K)\cong\A_{48}$.

After relabeling the direct factors, we may assume that $K=T_1\times\cdots\times T_{k-1}\cong\A_{47}^{k-1}$ and $G=K\times T_k$, where $T_k\cong\A_{47}$. The equality $\Core_{A/K}(G/K)=1$, together with $G/K\neq1$, implies that $G/K<A/K$. On the other hand, $A/K\leqslant\Aut(\Ga_K)\cong\A_{48}$ and $G/K\cong\A_{47}$. The subgroup $\A_{47}$ is maximal in $\A_{48}$, and so $A/K\cong\A_{48}$.

Let $C=C_A(K)$. The decomposition $G=K\times T_k$ gives $T_k\leqslant C$. The normality of $K$ in $A$ implies $C\unlhd A$, and hence $CK/K\unlhd A/K\cong\A_{48}$. The subgroup $CK/K$ contains $T_kK/K\cong\A_{47}$ and is nontrivial. The simplicity of $\A_{48}$ yields $CK/K=A/K$, and so $CK=A$. Moreover, $C\cap K=C_K(K)=Z(K)=1$. By the definition of $C$, the subgroups $C$ and $K$ centralize each other. It follows that $A=K\times C$ and $C\cong A/K\cong\A_{48}$.

The regularity of $G=K\times T_k$ on $V(\Ga)$ gives $|A_\a|=|A:G|=|C:T_k| =|\A_{48}:\A_{47}|=48$. By \cite[Section 2]{DM1980}, $\Ga$ is $(A,5)$-arc-regular. By the universal description of finite cubic 5-arc-regular groups in \cite{CL1989}, the group $A$ is a quotient of the universal group $G_5$. Thus there is an epimorphism $\theta:G_5\to A$. Since $k\geqslant2$, the group $K\cong\A_{47}^{k-1}$ contains a direct factor $T_1\cong\A_{47}$. Let $\pi_1:A=K\times C\to T_1$ be the natural projection and let $\varphi=\pi_1\theta$. Then $\varphi:G_5\to\A_{47}$ is an epimorphism.

Let $H_1\cong\A_{46}$ be a point stabilizer in the natural degree-47 action of $\A_{47}$ and let $H=\varphi^{-1}(H_1)$. Then $|G_5:H|=|\A_{47}:H_1|=47$. Since $H_1$ is core-free in $\A_{47}$, we have $\Core_{G_5}(H)=\ker\varphi$. Thus the permutation group induced by $G_5$ on the right cosets of $H$ is isomorphic to $G_5/\ker\varphi\cong\A_{47}$. Listing \ref{code:magma-5-arc} verifies in Magma \cite{magma} that no subgroup of index $47$ in $G_5$ has a coset action with image isomorphic to $\A_{47}$. This contradiction shows that $K<G$ is impossible. Hence $K=G$ and $G\unlhd A$. Thus $\Ga$ is a normal Cayley graph.
\end{proof}

\appendix

\section{Magma computations}\label{app:magma}

\subsection{\texorpdfstring{The factorization $\A_6=\A_5\Sy_3$}{The factorization A6=A5S3}}\label{app:A6}

We first verify that $\A_6$ admits no factorization $\A_6=\A_5\Sy_3$. Since $|\A_5||\Sy_3|=60\cdot6=360=|\A_6|$, such a factorization exists if and only if there are subgroups $H,K\leqslant\A_6$ with $H\cong\A_5$, $K\cong\Sy_3$ and $H\cap K=1$. The following Magma computation shows that no such pair exists.

\begin{lstlisting}[basicstyle=\ttfamily\small,breaklines=true,columns=fullflexible,caption={Verification that $\A_6$ admits no factorization $\A_6=\A_5\Sy_3$.},label={lst:A6-factorization}]
A6 := AlternatingGroup(6);
A5 := AlternatingGroup(5);
S3 := SymmetricGroup(3);
subs60 := [ R`subgroup : R in Subgroups(A6 : OrderEqual := 60) ];
subs6 := [ R`subgroup : R in Subgroups(A6 : OrderEqual := 6) ];
fac := [ <H,K> : H in subs60, K in subs6 | IsIsomorphic(H,A5) and IsIsomorphic(K,S3) and Order(H meet K) eq 1 ];
#fac;
// Output:
// 0
\end{lstlisting}

Thus $\A_6$ has no factorization $\A_6=\A_5\Sy_3$.

\subsection{\texorpdfstring{The universal group $G_5$}{The universal group G5}}\label{app:G5}

We verify that the universal group $G_5$ for cubic $5$-arc-regular graphs has no quotient isomorphic to $\A_{47}$. We use the presentation of $G_5$ given in \cite{CL1989}. Suppose that there is an epimorphism $\varphi:G_5\to\A_{47}$. Let $H_1\cong\A_{46}$ be a point stabilizer in the natural degree-$47$ action of $\A_{47}$ and let $H=\varphi^{-1}(H_1)$. Then $|G_5:H|=47$, and the coset action of $G_5$ on the right cosets of $H$ has image $G_5/\Core_{G_5}(H)\cong\A_{47}$. It is enough to consider the coset actions arising from subgroups of index $47$ in $G_5$. The following computation was carried out in Magma \cite{magma}.

\begin{lstlisting}[basicstyle=\ttfamily\small,breaklines=true,columns=fullflexible,caption={Verification that $G_5$ has no quotient isomorphic to $\A_{47}$.},label={code:magma-5-arc}]
F<h,p,q,r,s,a> := FreeGroup(6);
G5 := quo< F | h^3, p^2, q^2, r^2, s^2, a^2, p*q*p^-1*q^-1, p*r*p^-1*r^-1, p*s*p^-1*s^-1, q*r*q^-1*r^-1, q*s*q^-1*s^-1, p*q*(r*s)^2, h^-1*p*h*p^-1, h^-1*q*h*r, h^-1*r*h*p*q*r, (s*h)^2, a*p*a*q, a*r*a*s >;
L := LowIndexSubgroups(G5, <47,47>);
images := [ CosetImage(G5,H) : H in L ];
orderA47 := Factorial(47) div 2;
hasA47quotient := exists{ P : P in images | Order(P) eq orderA47 };
hasA47quotient;
// Output:
// false
\end{lstlisting}

The command \texttt{LowIndexSubgroups} returns representatives of the conjugacy classes of subgroups of index 47 in $G_5$, and \texttt{CosetImage} constructs the image of each associated coset action. The final output is \texttt{false}, so none of these images has order $|\A_{47}|=47!/2$. Hence no degree-47 coset action has image isomorphic to $\A_{47}$, and $G_5$ has no quotient isomorphic to $\A_{47}$.

\vskip 10pt

\noindent{\bf Declaration of competing interest}

The authors declare that they have no known competing financial interests or personal relationships that could have appeared to influence the work reported in this paper.
The authors sincerely thank Dr. Binzhou Xia for their valuable comments and suggestions.

\end{document}